\documentclass[10pt,reqno]{amsart}

\title{A Torelli theorem for graphs via quasistable divisors}

\author{Alex Abreu and Marco Pacini}

\usepackage{amsmath}
\usepackage{amsthm}
\usepackage{amsfonts}
\usepackage{mathrsfs}
\usepackage{amssymb}
\usepackage{amscd}
\usepackage[all]{xy}
\usepackage{pgf}
\usepackage{tikz}
\usetikzlibrary{decorations.pathmorphing, decorations.markings, calc, cd}

\usepackage{amscd}

\usepackage{enumerate}
\usepackage{hyperref}

    \newtheorem{Lem}{Lemma}[section]
    \newtheorem{Lem-Def}{Lemma-Definition}[section]
    \newtheorem{Prop}[Lem]{Proposition}
       
    \newtheorem*{thm}{Theorem}

    \newtheorem{Thm}[Lem]{Theorem}  
    \newtheorem{Cor}[Lem]{Corollary}
		\newtheorem*{Thm*}{Theorem}

\theoremstyle{definition}

    \newtheorem{Def}[Lem]{Definition}

    \newtheorem{Rem}[Lem]{Remark}

\newcommand{\ra}{\rightarrow}

\newcommand{\E}{\mathcal E}

\newcommand{\col}{\colon}

\newcommand{\ol}{\overline}

\newcommand{\wt}{\widetilde}

\newcommand{\QS}{\mathbf{QD}}

\newcommand{\PosFix}{\mathbf P}
\newcommand{\PosFixx}{\mathbf R}

\DeclareMathOperator{\bridges}{Br}
\DeclareMathOperator{\ND}{ND}
\DeclareMathOperator{\rk}{rk}
\DeclareMathOperator{\val}{val}
\DeclareMathOperator{\Div}{Div}

\DeclareMathOperator{\Del}{Del}

\DeclareMathOperator{\can}{can}

\DeclareMathOperator{\quasi}{qs}

\begin{document}

\maketitle

\begin{abstract}
    The Torelli theorem establishes that the Jacobian of a smooth projective curve, together with the polarization provided by the theta divisor, fully characterizes the curve.
In the case of nodal curves, there exists a concept known as fine compactified Jacobian. The fine compactified Jacobian of a curve comes with a natural stratification that can be regarded as a poset. Furthermore, this poset is entirely determined by the dual graph of the curve and is referred to as the poset of quasistable divisors on the graph.
We present a combinatorial version of the Torelli theorem, which demonstrates that the poset of quasistable divisors of a graph completely determines the biconnected components of the graph (up to contracting separating edges). Moreover, we achieve a natural extension of this theorem to tropical curves.
\end{abstract}


\section{Introduction}

The classical Torelli theorem states that if $C$ and $C'$ are two genus $g$ smooth projective curves whose Jacobian varieties are isomorphic (as principally polarized abelian varieties), then $C$ and $C'$ are isomorphic. 
For nodal curves, a variant of the Torelli theorem emerges considering compactified Jacobians. 
In \cite{CVCurves}, Caporaso and Viviani proved that a stable curve can be reconstructed from its Caporaso compactified Jacobian and theta divisor, provided that its dual graph is $3$-edge connected. We refer to \cite{CapJac} for the construction of the compactified Jacobian and to \cite{CapTheta} for a study of the theta divisor of the compactified Jacobian. The main result in \cite{CVCurves} is based on a previous combinatorial result proved in \cite{CVGraphs}, stating that it is possible to reconstruct a  graph from its Albanese variety, provided the graph is $3$-vertex connected (this resolved a question posed in \cite{BHN}), see also \cite{Art}.\par

 More general results are also proved in \cite{CVGraphs} and \cite{CVCurves}: two stable curves without separating nodes have isomorphic compactified Jacobians toghether with theta divisors if and only if the curves are $C1$-equivalent (see \cite[Definition 2.1.5]{CVCurves} for the definition of $C1$-equivalence). The general statement for graphs is: two graphs without bridges have isomorphic Albanese varieties if and only if the graphs are cyclically equivalent.  The observation connecting the two results is that if the compactified  Jacobians of two stable curves are isomorphic, then the Albanese varieties of the dual graphs of the curves are isomorphic as well.\par 

 The question that motivated this paper is: 

 \smallskip

 \begin{equation}
 \label{eq:question}
\text{\emph{can one get a more refined Torelli theorem by considering other compactified Jacobians?}}
\end{equation}

\smallskip

In this paper we answer a combinatorial version of the above  question. We consider Esteves compactified Jacobian of a nodal curve, parametrizing quasistable torsion-free rank-1 sheaves of fixed degree on a curve, constructed in \cite{EE01}. Both Caporaso and Esteves compactified Jacobians for a nodal curve are instances of Oda-Seshadri construction of compactified Jacobians constructed in \cite{OS} (see \cite{Alexeev} and \cite[Section 6]{EE01}).

 
 In \cite{CVGraphs}, a crucial ingredient in the proof of Torelli theorem for graphs is the Delaunay decomposition $\Del(\Gamma)$ of a graph $\Gamma$ and its associated poset (i.e., partially ordered set) $\overline{\mathcal{OP}}_\Gamma$. The poset $\overline{\mathcal{OP}}_{\Gamma}$ is the poset encoding the natural stratification of the Caporaso compactified Jacobian of a curve with dual graph $\Gamma$ (see \cite[Lemma 4.1.6]{CVGraphs}). For a $3$-edge connected graph $\Gamma$, the Delaunay decomposition $\Del(\Gamma)$  determines and is determined by the poset $\overline{\mathcal{OP}}_\Gamma$. The key results are that the Albanese variety of a graph determines its Delaunay decomposition and, if the graph is $3$-edge connected, the Delaunay decomposition only depends from the cyclic equivalence class of the graph. The general statement of this result can be found in \cite[Theorem 5.3.2]{CVGraphs}.  
  
The Esteves compactified Jacobian exhibits a natural stratification that can be viewed as a poset. This is the poset $\QS(\Gamma)$ of quasistable (pseudo-)divisors of degree $g-1$ on the dual graph $\Gamma$ of the curve, which corresponds to the multidegrees of quasistable torsion-free rank-1 sheaves of degree $g-1$ on the curve. In this paper we prove that this poset plays a crucial role in characterizing the nodal curve. Remarkably, the poset structure  entirely determines  the dual graph of the curve. Thus, by studying the poset of quasistable divisors, one can gain insights into the topology and combinatorial properties of the curve itself.\par

Note worthy, the poset $\QS(\Gamma)$ is the poset induced by a refinement of the Delaunay decomposition $\Del(\Gamma)$ of $\Gamma$. This refinement holds more combinatorial information about the graph than the Delaunay decomposition. Hence, it is expected a more refined Torelli theorem for graphs using the poset $\QS(\Gamma)$. The main theorem of this paper is the following result. 

\begin{thm}[Theorem \ref{thm:main1}]
Let $\Gamma$ and $\Gamma'$ be graphs with set of bridges $\bridges(\Gamma)$ and $\bridges(\Gamma')$. The posets $\QS(\Gamma)$ and $\QS(\Gamma')$ are isomorphic if and only if there is a bijection between the biconnected components of $\Gamma/\bridges(\Gamma)$ and $\Gamma'/\bridges(\Gamma')$ such that the corresponding components are isomorphic as pure graphs. 
\end{thm}

In particular, a pure biconnected graph $\Gamma$ can be reconstructed from its poset $\QS(\Gamma)$. Hence, for pure biconnected graphs, we get a more refined Torelli theorem. Indeed, there are nonisomorphic $3$-edges connected biconnected graphs $\Gamma$ and $\Gamma'$ that are cyclic equivalent, and hence, by the result of Caporaso and Viviani, the poset $\overline{\mathcal{OP}}_\Gamma$ and $\overline{\mathcal{OP}}_{\Gamma'}$ are isomorphic, while $\QS(\Gamma)$ and $\QS(\Gamma')$ are not.

As a byproduct, we get a Torelli theorem for tropical curves. We prove that the tropical Jacobian $J(X)$ of a tropical curve $X$, together with its decomposition via quasistable divisors, determines the biconnected components of the  tropical curve. 

\begin{thm}[Theorem \ref{thm:main2}]
    Let $X$ and $X'$ be tropical curves without bridges such that $J(X)$ and $J(X')$ are isomorphic as polyhedral complexes (with the structure of polyhedral complexes given by the poset of quasistable divisor on the underlying graph). There is a bijection between the biconnected components of $X$ and $X'$ such that the corresponding components are isomorphic.
\end{thm}


We conclude this introduction with some remarks regarding Question \eqref{eq:question}. The combinatorial result provided by Theorem \ref{thm:main1} implies that a geometric Torelli Theorem utilizing fine compactified Jacobians should be distinct and potentially more refined than the result obtained by Caporaso and Viviani in their work \cite{CVCurves}. 
So far we did not found examples of curves with no separating nodes whose fine compactified Jacobian are isomorphic (together with the theta divisor).

\section{Preliminaries}

\subsection{Posets}

 In this paper we will only consider finite posets. Given a poset $(P,\leq_P)$ and a subset $S\subset P$, the \emph{induced partial order} $\leq_S$ on $S$ is given by $x\leq_Sy$ for $x,y\in S$ if and only if $x\leq_P y$ in $P$. We refer to $(S,\leq_S)$ as the \emph{induced subposet}.

A \emph{lower set} of a poset $(P,\leq_P)$ is a set $U\subset P$ such that whenever $x\in U$ and $y\leq_Px$, then $y\in U$. We define a topology on the poset $P$ where the closed subsets are the lower sets.

We say that an element $x$ \emph{covers} an element $y$ of $P$ if $x>_P y$ and there are no $z\in P$ such that $x<_P z<_P y$.
A poset is called \emph{ranked} if all the maximal chains have the same length. A ranked poset $P$ comes equipped with a rank function $\rk\col P\to \mathbb{Z}$ such that $\rk(x)=\rk(y)+1$ whenever $x$ covers $y$ and $\rk(x)=0$ whenever $x$ is a minimal element of $P$. The \emph{Hasse diagram} of a poset is the oriented graph whose vertices are the elements of $P$ and oriented edges are from $x$ to $y$ whenever $y$ covers $x$. \par 
  A \emph{morphism} between posets $P$ and $P'$ is an order-preserving function (or, equivalently, a  continuous function) $f\col P\to P'$. Moreover, we say that $f$ \emph{preserves the cover relations} if $f(x)$ covers $f(y)$ whenever $x$ covers $y$, for $x,y\in P$. If $P$ and $P'$ are ranked, then we say that $f$ is a \emph{morphism of ranked posets} if $\rk(f(x))=\rk(x)$ for every $x\in P$. An \emph{isomorphism} of posets is a morphism of posets admitting an inverse morphism. As usual, a morphism of posets is closed if it takes closed subsets to closed subsets. 
  
  \begin{Rem}\label{rem:closed}
  Notice that $f\col P\ra P'$ is a closed morphism of posets, if and only if, for any $x\in P$ and $y'\in P'$ such that $y'\leq_{P'} f(x)$  there exists $y\in P$ such that $y\leq_P x$ and $f(y)=y'$.
\end{Rem}

\subsection{Graphs}

Let $\Gamma$ be a graph. We denote by $V(\Gamma)$ and $E(\Gamma)$  the sets of vertices and edges of $\Gamma$, and $w_\Gamma\col V(\Gamma)\ra \mathbb Z_{\ge0}$ the weight function of $\Gamma$. A graph is \emph{pure} if $w_\Gamma(v)=0$ for every $v\in V(\Gamma)$.  
Given a subset $V\subset V(\Gamma)$, we set $V^c:=V(\Gamma)\setminus V$.  For subsets $V,W\subset V(\Gamma)$, we define $E(V,W)$ as the set of edges of $\Gamma$ connecting a vertex in $V$ with a vertex in $W$. In particular, $E(V,V)$ is the set of edges connecting two (possibly coinciding) vertices of $V$. We set $\delta_V=|E(V,V^c)|$. We also denote by $\Gamma(V)$ the subgraph of $\Gamma$ whose set of vertices is $V$ and whose set of edges is $E(V,V)$. The edges $e_1,e_2\in E(\Gamma)$ are \emph{parallel} if there are two vertices incident to both $e_1$ and $e_2$. An \emph{end-vertex} of an edge $e$ is a vertex which is incident to $e$.

For a vertex $v\in V(\Gamma)$, we let $E(v)$ be the set of edges of $\Gamma$ that are incident to $v$. Moreover, we let $\Gamma\setminus\{v\}$ be the subgraph of $\Gamma$ with set of vertices equal to $V(\Gamma)\setminus\{v\}$ and set of edges equal to $E(\Gamma)\setminus E(v)$. For a subset $\E\subset E(\Gamma)$ and a vertex $v\in V(\Gamma)$, we define $\val_\E(v)$ to be the number of edges of $\E$ incident to $v$, with loops counted twice. We set $\val(v):=\val_{E(\Gamma)}(v)$ which is called the \emph{valence} of $v$ in $\Gamma$.

A \emph{cut} of $\Gamma$ is a subset $\mathcal E\subset E(\Gamma)$ such that $\mathcal E=E(V,V^c)$, for some subset $V\subset V(\Gamma)$. A \emph{bond} of $\Gamma$ is a minimal cut of $\Gamma$. A \emph{hemisphere} of $\Gamma$ is a subset $V\subset V(\Gamma)$ such that $\Gamma(V)$ and $\Gamma(V^c)$ are connected subgraphs of $\Gamma$. Equivalently, $V$ is a hemisphere if and only if $E(V,V^c)$ is a bond. 
The \emph{genus} of $\Gamma$ is defined as $g_\Gamma:=b_1(\Gamma)+\sum_{v\in V(\Gamma)} w_\Gamma(v)$, where $b_1(\Gamma)$ is the first Betti number of $\Gamma$. For every subset $V\subset V(\Gamma)$, we let 
$g_V$ be the genus of the 
graph $\Gamma(V)$. In particular, we have $g_{V(\Gamma)}=g_\Gamma$.

A \emph{cycle} of the graph $\Gamma$ is a subset $\gamma\subset E(\Gamma)$ such that there is a connected subgraph of $\Gamma$ whose edges are the elements of $\gamma$ and whose vertices (called the \emph{vertices of the cycle}) have all valence 2. The graph $\Gamma$ is a \emph{tree} if it is connected and has no cycles. Equivalently, $\Gamma$ is a tree if and only if $b_1(\Gamma)=0$. A \emph{spanning tree} of $\Gamma$ is a connected subgraph of $\Gamma$ which is a tree and whose set of vertices is equal to $V(\Gamma)$. We usually see a spanning tree as a subset $T\subset E(\Gamma)$. We will call the complement of a spanning tree (in $E(\Gamma))$ a \emph{maximally nondisconnecting} subset of $\Gamma$.

 A \emph{cyclic equivalence} between two graphs  $\Gamma$ and $\Gamma'$ is a bijection $E(\Gamma)\ra E(\Gamma')$ that induces a bijection between the cycles of $\Gamma$ and the cycles of $\Gamma'$. 

\begin{Rem}
\label{rem:cyclic_equiv_trees}
    Given a bijection $f\colon E(\Gamma)\to E(\Gamma')$, the following conditions are equivalent.
    \begin{enumerate}
        \item The bijection $f$ is a cyclic equivalence.
        \item The bijection $f^{-1}$ is a cyclic equivalence. 
        \item The bijection $f$ induces a bijection between the set of spanning trees of $\Gamma$ and $\Gamma'$.
        \item The bijection $f$ induces a bijection between the set of bonds of $\Gamma$ and $\Gamma'$.
        \item The bijection $f$ induces a bijection between the set of cuts of $\Gamma$ and $\Gamma'$.
    \end{enumerate}
\end{Rem}

An edge $e$ of $\Gamma$ is called a \emph{bridge} if $\Gamma$ becomes disconnected after the removal of $e$. We let $\bridges(\Gamma)$ be the set of bridges of $\Gamma$.
We denote the set of nondisconnecting edges of $\Gamma$ by 
\begin{equation}\label{eq:ND}
\ND(\Gamma):=E(\Gamma)\setminus \bridges(\Gamma).
\end{equation}

A \emph{weakly cyclic equivalence} between two graphs $\Gamma$ and $\Gamma'$ is a bijection $f\col \ND(\Gamma)\to \ND(\Gamma')$ that induces a bijection between the cycles of $\Gamma$ and the cycles of $\Gamma'$ (recall that every cycle of $\Gamma$ is contained in $\ND(\Gamma)$). Equivalently, a weakly cyclic equivalence is a cyclic equivalence between $\Gamma/\bridges(\Gamma)$ and $\Gamma'/\bridges(\Gamma')$

\begin{Rem}
\label{rem:weak_cyclic_equiv_trees}
    Given a bijection $f\colon \ND(\Gamma)\to \ND(\Gamma')$, the following conditions are equivalent.
    \begin{enumerate}
        \item The bijection $f$ is a weakly cyclic equivalence.
        \item The bijection $f^{-1}$ is a weakly cyclic equivalence. 
        \item The bijection $f$ induces a bijection between the sets of maximally nondisconnectig subsets of  $\Gamma$ and $\Gamma'$.
    \end{enumerate}
\end{Rem}

A \emph{subdivision} of the graph $\Gamma$ is a graph obtained from $\Gamma$ inserting a number $n_e\ge0$ of vertices in the interior of every edge $e\in E(\Gamma)$. 
We say that $\Gamma$ is \emph{biconnected} if, for every subdivision $\widehat{\Gamma}$ of $\Gamma$, the removal of any vertex of $\widehat{\Gamma}$ does not disconnect the graph $\widehat{\Gamma}$. In particular, a graph with exactly one edge is biconnected if and only if it is a loop. Otherwise, a graph with at least two edges is  biconnected if and only if any two vertices of the graph are vertices of a cycle of the graph. Of course, if $\Gamma$ has a bridge, then $\Gamma$ is not  biconnected.  A \emph{biconnected component} of $\Gamma$ is a maximal biconnected subgraph of $\Gamma$. Every graph admits a unique decomposition into biconnected components. An \emph{articulation vertex} of $\Gamma$ is a vertex of $\Gamma$ such that the removal of $v$ disconnects the graph. 

Consider a subset $\mathcal E$ of $E(\Gamma)$. We denote by $\Gamma_\E$ the graph obtained from $\Gamma$ by removing the edges in $\E$, with $E(\Gamma_\E)=E(\Gamma)\setminus \E$ and $V(\Gamma_\E)=V(\Gamma)$. We also denote by $\Gamma^\E$ the subdivision of $\Gamma$ obtained from $\Gamma$ by inserting exactly one vertex, called \emph{exceptional} and denoted by $v_e$, in the interior of every edge $e\in \E$. We have $V(\Gamma^\E)=V(\Gamma)\cup \{v_e;e\in \E\}$. Finally, we let $\Gamma/\mathcal E$ the graph obtained by the contraction of the edges in $\mathcal E$. In this case, we say that $\Gamma$ \emph{specializes} to $\Gamma/\E$, and we write $\iota\col\Gamma\ra \Gamma/\E$. Notice that we have an induced surjective function $\iota\col V(\Gamma)\ra V(\Gamma/\E)$ and an inclusion  $E(\Gamma/\E)=E(\Gamma)\setminus \E\stackrel{\iota}{\ra} E(\Gamma)$. The case in which   $\E=\bridges(\Gamma)$ will play an important role later on. It is clear that $\Gamma/\bridges(\Gamma)$ is a graph without bridges.

\subsection{Divisors on graphs}

Let $\Gamma$ be a graph.
A \emph{divisor} $D$ on $\Gamma$ is a formal sum $D=\sum_{v\in V(\Gamma)}D(v)v$, where $D(v)\in \mathbb Z$. We denote by $\Div(\Gamma)$ the abelian group of divisors of $\Gamma$. For every subset $V\subset V(\Gamma)$, we set $D(V)=\sum_{v\in V}D(v)$. The \emph{degree} of a divisor $D$ is the integer $D(V(\Gamma))$. A \emph{pseudo-divisor} on $\Gamma$ is a pair $(\E,D)$, where $\E$ is a subset of $E(\Gamma)$ and $D$ is a divisor on $\Gamma^\E$ such that $D(v_e)=1$, for every $e\in \E$. The \emph{degree} of a pseudo-divisor $(\E,D)$ is the degree of the divisor $D$. Given a pseudo-divisor $(\E,D)$ on $\Gamma$, we set
\begin{equation}\label{eq:eps-delta}
\epsilon_{\Gamma}(\E,D)=\E
\;\; \text{ and }\;\;
\delta_{\Gamma}(\E,D)=D.
\end{equation}

If $\widehat{\Gamma}$ is a subdivision of a graph $\Gamma$, we can extend a divisor $D$ on $\Gamma$ to a divisor on $\widehat{\Gamma}$, setting $D(v)=0$ for every $v\in V(\widehat{\Gamma})\setminus V(\Gamma)$. Thus for every pseudo-divisor $(\E,D)$ on $\Gamma$, we could see $D$ as a divisor on the subdivision $\Gamma^{E(\Gamma)}$ of $\Gamma$. In particular, given pseudo-divisors $(\E_1,D_1)$ and $(\E_2,D_2)$, the sum  $D_1+D_2$ will make sense as a sum of divisors on $\Gamma^{E(\Gamma)}$.

Let $\iota\col \Gamma\ra \Gamma'$ be a specialization of graphs. Given a divisor $D$  on $\Gamma$, we have an induced divisor  $\iota_*(D)$ on $\Gamma'$ such that $\iota_*(D)(v')=\sum_{v\in \iota^{-1}(v')} D(v)$, for every $v'\in V(\Gamma')$. Notice that, if $\E$ is a subset of $E(\Gamma)$, then we have an induced specialization $\iota^\E\col \Gamma^\E\ra \Gamma'^{\E'}$, where $\E'=\E\cap E(\Gamma')$.
 Therefore, if $(\E,D)$ is a pseudo-divisor on $\Gamma$, we have an induced pseudo-divisor $\iota_*(\E,D):=(\E',\iota^\E_*(D))$ on $\Gamma'$. Given pseudo-divisors $(\E, D)$ on $\Gamma$ and $(\E', D')$ on $\Gamma'$, we say that $(\Gamma, \E, D)$ \emph{specializes} to $(\Gamma', \E', D')$ if the following conditions hold
\begin{enumerate}
    \item there is a specialization $\iota\col \Gamma\ra \Gamma'$ such that $\E'\subset \E\cap E(\Gamma')$;
    \item there is a specialization $\iota^\E\col \Gamma^\E\ra \Gamma'^{\E'}$
 such that $\iota^\E_*(D)=D'$;
    \item the following diagrams are commutative 
    \[
\SelectTips{cm}{11}
\begin{xy} <16pt,0pt>:
\xymatrix{ V(\Gamma)\ar[d]\ar[r]^{\iota} &V(\Gamma')\ar[d]&& E(\Gamma')\ar[r]^\iota &E(\Gamma)\\
             V(\Gamma^\E)\ar[r]^{\iota^\E}& V(\Gamma'^{\E'})&&E(\Gamma'^{\E'})\ar[u]\ar[r]^{\iota^\E} & E(\Gamma^\E)\ar[u]
}
\end{xy}
\]
\end{enumerate}

If $(\Gamma, \E, D)$ specializes to $(\Gamma', \E', D')$, we write $(\Gamma,\E,D)\ra (\Gamma',\E',D')$. If $\Gamma=\Gamma'$ and $\iota$ is the identity, we simply write $(\E,D)\ra (\E',D')$.

 An \emph{elementary specialization} is a specialization of type $(\E,D)\ra (\E',D')$, where $|\E'|=|\E|-1$. In this case, we have $\E'=\E\setminus\{e\}$ for some edge $e\in E(\Gamma)$, and we say that the elementary specialization is \emph{over $e$}. Notice that every specialization is a composition of elementary specializations.

\begin{Rem}\label{rem:elementary}
 Let $(\E,D)$ be a pseudo-divisor on $\Gamma$ and consider $e\in \E$. 
  If $e$ is not a loop with end-vertices $s,t$, then $(\E,D)\ra (\E\setminus\{e\},D-v_e+s)$ and $(\E,D)\ra(\E\setminus\{e\},D-v_e+t)$ are all the elementary specializations over $e$ having  $(\E,D)$ as source. If $e$ is a loop of $\Gamma$ with end-vertex $s$, then $(\E,D)\ra (\E\setminus\{e\},D-v_e+s)$ is the unique elementary specialization over $e$ having $(\E,D)$ as source.
 Notice that if $(\E,D_1)$ and $(\E,D_2)$ both specialize to the same pseudo-divisors $(\E\setminus\{e\},D'_1)$ and $(\E\setminus\{e\},D'_2)$, with $D'_1\ne D'_2$, then $D_1=D_2$.
\end{Rem}

A \emph{polarization} on the graph $\Gamma$ is a function $\mu\col V(\Gamma)\ra \mathbb R$ such that $\sum_{v\in V(\Gamma)} \mu(v)\in\mathbb Z$. For every subset $V\subset V(\Gamma)$, we set $\mu(V)=\sum_{v\in V} \mu(v)$. The \emph{degree} of a polarization $\mu$ is the integer $\mu(V(\Gamma))$. Given a specialization of graphs $\iota\col \Gamma\ra \Gamma'$ and a polarization $\mu$ on $\Gamma$ of degree $d$, we have an induced polarization $\iota_*(\mu)$ on $\Gamma'$ of degree $d$ given by $\iota_*(\mu)(v')=\sum_{v\in \iota^{-1}(v')} \mu(v)$. Given a subset $\E\subset E(\Gamma)$ and a degree $d$ polarization $\mu$ on $\Gamma$, we have an induced polarization $\mu^\E$ on $\Gamma^\E$ of degree $d$ given by $\mu^\E(v)=\mu(v)$ if $v\in V(\Gamma)$, and $\mu^\E(v)=0$ if $v\in V(\Gamma^\E)\setminus V(\Gamma)$. We also have an induced polarization $\mu_\E$ of degree $d-|\E|$ on $\Gamma_\E$ taking $v\in V(\Gamma_\E)$ to  $\mu_\E(v)=\mu(v)-\frac{1}{2}\val_\E(v)$.

Let $v_0$ be a vertex on the graph $\Gamma$ and $\mu$ a polarization on $\Gamma$ of degree $d$. Let $D$ be a divisor on $\Gamma$ of degree $d$. For every subset $V\subset V(\Gamma)$, we set
\begin{equation}
\beta_{\Gamma,D}(V):=D(V)-\mu(V)+\frac{\delta_V}{2}.
\end{equation}
We say that $D$ is \emph{$(v_0,\mu)$-quasistable} if $\beta_{\Gamma,D}(V)\ge0$ for every $V\subset V(\Gamma)$, with strict inequality if $v_0\not\in V$. 

\begin{Rem}\label{rem:hemi}
    To check that a divisor is $(v_0,\mu)$-quasistable, it suffices to check the condition of $(v_0,\mu)$-quasistability for all hemispheres of $\Gamma$.
\end{Rem}

\begin{Rem}
The definition of pseudo-divisor in this paper is different from the one given in \cite{APPLMS}, where a pseudo-divisor has degree $-1$ on every exceptional vertex. As a consequence, we have to change the definition of the induced polarization $\mu_\E$ and the notion of quasistability (which usually requires that the inequality is strict if $v_0\in V$). All the result of the paper could be proved in both setup. The reason why we preferred the new setup is because of Lemma  \ref{lem:tree}.
\end{Rem}

Given a pseudo-divisor $(\E,D)$ of degree $d$ on the graph $\Gamma$, we say that $(\E,D)$ is \emph{$(v_0,\mu)$-quasistable} if the divisor $D$ on $\Gamma^\E$ is $(v_0,\mu^\E)$-quasistable.

The \emph{canonical polarization of degree $g-1$} on the graph $\Gamma$ is the polarization $\mu_{\can}$ of degree $g-1$ such that
\begin{equation}\label{eq:canonical-pol}
\mu_{\can}(V) = g_V-1+\frac{\delta_V}{2},
\end{equation}
for every hemisphere $V\subset V(\Gamma)$. 
In this case, if $(\E,D)$ is a pseudo-divisor on $\Gamma$, then for every hemisphere $V\subset V(\Gamma^{\E})$ we have 
\begin{equation}\label{eq:integer}
\beta_{\Gamma^\E,D}(V)=D(V)-\mu^\E_{\can}(V)+\frac{\delta_V}{2}=D(V)-g_V+1,
\end{equation}
(recall that $D$ is a divisor on $\Gamma^\E$).
Given a $(v_0,\mu_{\can})$-quasistable pseudo-divisor $(\E,D)$ on $\Gamma$, we simply say that $(\E,D)$ is \emph{$v_0$-quasistable}. 

\begin{Rem}\label{rem:canonical-pol}
If $\E\subset E(\Gamma)$ is a nondisconneting subset of $E(\Gamma)$, then $(\mu_{\can})_\E$ is the canonical polarization of $\Gamma_\E$. \par 
\end{Rem}

\section{The poset of quasistable divisors}

Let $\Gamma$ be a graph. 
Given a vertex $v_0$ and a polarization $\mu$ on $\Gamma$, the set $\QS_{v_0,\mu}(\Gamma)$ of $(v_0,\mu)$-quasistable pseudo-divisors on $\Gamma$ forms a poset, where $(\E,D)\ge (\E',D')$ if there is a specialization $(\E,D)\ra (\E',D')$.  Given a subset $\E\subset E(\Gamma)$, we let 
\begin{equation}\label{eq:QDE}
\QS_{v_0,\mu}(\Gamma,\E)=\{D\in \Div(\Gamma^{\E}) ; (\E,D)\in \QS_{v_0,\mu}(\Gamma)\}.    
\end{equation}
The poset  $\QS_{v_0,\mu}(\Gamma)$ is ranked, with rank function taking a pseudo-divisor $(\E,D)$ to $|\E|$. We call $|\E|$ the \emph{rank} of the pseudo-divisor $(\E,D)$.

\begin{Rem}\label{rem:iso-bridge}
Let $\Gamma$ be a graph, $v_0$ a vertex of $\Gamma$, and $\mu$ a polarization on $\Gamma$.
If $e$ is a bridge of $\Gamma$ and $\iota\col \Gamma\ra \Gamma/\{e\}$ is the contraction of $e$, then $\QS_{v_0,\mu}(\Gamma)$ is naturally isomorphic to $\QS_{\iota(v_0),\iota_*(\mu)}(\Gamma/\{e\})$. Therefore, if we consider the specialization $\iota\col \Gamma\ra \Gamma/\bridges(\Gamma)$, then we have a natural isomorphism
\[
\QS_{v_0,\mu}(\Gamma)\cong\QS_{\iota(v_0),\iota_*(\mu)}(\Gamma/\bridges(\Gamma)).
\]
\end{Rem}

\begin{Rem}
\label{rem:AP_46}

Let $\Gamma$ be a graph, $\mu$ a polarization on $\Gamma$ and $\E\subset E(\Gamma)$ a  subset. 
The following properties are consequences of 
 \cite[Proposition 4.6]{APPLMS}.
 \begin{enumerate}
\item If $(\E,D)\in \QS_{v_0,\mu}(\Gamma)$ then $\E\subset \ND(\Gamma)$ (recall Equation \eqref{eq:ND}).
\item If $(\E,D)$ is a $(v_0,\mu)$-quasistable divisor on $\Gamma$ and $\iota\col \Gamma\ra \Gamma'$ is a specialization, then $\iota_*(\E,D)$ is a $(\iota(v_0),\iota_*(\mu))$-quasistable pseudo-divisor on $\Gamma'$.
 \end{enumerate}

 If $\E\subset E(\Gamma)$ is nondisconnecting, then:
 \begin{enumerate}
 \setcounter{enumi}{2}
 \item We have a natural inclusion $\QS_{v_0,\mu_\E}(\Gamma_\E)\subset \QS_{v_0,\mu}(\Gamma)$, taking a pseudo-divisor $(\E',D')$ to the pseudo-divisor $(\E\cup \E', D'+\sum_{e\in \E}v_e)$. Moreover, for every $S\subset E(\Gamma)\setminus \E$, we can identify $\QS_{v_0,\mu_\E}(\Gamma_\E,S)$  with $\QS_{v_0,\mu}(\Gamma,\E\cup S)$.
 \item If $\mu=\mu_{\can}$, then we have an inclusion $\QS_{v_0}(\Gamma_\E)\subset \QS_{v_0}(\Gamma)$ (combine Remark \ref{rem:canonical-pol} and  item (1)).
 \item If $\E$ is a maximally nondisconnecting subset of $\Gamma$, then $\QS_{v_0,\mu_\E}(\Gamma_\E)$ is a singleton. 
 \item The maximal elements of $\QS_{v_0,\mu}(\Gamma)$ are of the form $(\E,D)$ where $\E$ is a maximally nondisconnecting subset of $\Gamma$. 
 \item For each maximally nondisconnecting subset $\E$ of $\Gamma$, there exists exactly one $D\in \QS_{v_0,\mu}(\Gamma, \E)$. In particular, the number of maximal elements of $\QS_{v_0,\mu}(\Gamma)$ is equal to the number of spanning trees of $\Gamma$.
 \end{enumerate}
\end{Rem}


Let $\Gamma$ be a graph, $v_0$ a vertex of $\Gamma$, and $\mu$ a polarization on $\Gamma$. Two pseudo-divisors $(\E,D)$ and $(\E,D')$ in $\QS_{v_0,\mu}(\Gamma)$ are \emph{upper-connected} in $\QS_{v_0,\mu}(\Gamma)$ if there are edges $e_i\in E(\Gamma)\setminus \E$ for $i=1,\dots,n$, divisors $D_i$ on $\Gamma^{\E\cup \{e_i\}}$ for $i=1,\dots,n$, and divisors $D'_i$ on $\Gamma^{\E}$ 
 for $i=0,\dots,n$
 such that the following conditions hold
\begin{enumerate}
    \item we have that $D_i\in \QS_{v_0,\mu}(\Gamma, \E\cup\{e_i\})$ for $i=1,\dots,n$ and $D'_i\in \QS_{v_0,\mu}(\Gamma,\E)$
    for $i=0,\dots,n$;
    \item we have $(\E,D)=(\E,D_{0}')$ and $(\E,D')=(\E,D_n')$;
    \item we have 
 $(\E,D_{i-1}')\leq (\E\cup \{e_i\},D_i)$
 and $(\E,D'_i)\le (\E\cup \{e_i\},D_i)$
 for $i=1,\dots,n$.
\end{enumerate}

\begin{Prop}\label{prop:uppper-connected}
Let $\Gamma$ be a graph, $v_0$ a vertex of $\Gamma$, and $\mu$ a polarization on $\Gamma$. Consider divisors $D,D'\in \QS_{v_0,\mu}(\Gamma,\E)$, for some subset $\E\subset E(\Gamma)$. Then $(\E,D)$ and $(\E,D')$ are upper-connected in $\QS_{v_0,\mu}(\Gamma)$.
\end{Prop}
\begin{proof}
As recalled in Remark \ref{rem:AP_46}, 
we have an inclusion $\QS_{v_0,\mu_\E}(\Gamma_{\E})\subset \QS_{v_0,\mu}(\Gamma)$. Hence we can assume $\E=\emptyset$. We will proceed by induction on the number of edges of $\Gamma$. If $\Gamma$ has only one edge the result is clear. Otherwise, fix an edge $e\in E(\Gamma)$ and consider the contraction $\iota\col \Gamma\to \Gamma/\{e\}$ of $e$. Recall that the map $\iota_*\colon \QS_{v_0,\mu}(\Gamma)\to \QS_{\iota(v_0),\iota_*(\mu)}(\Gamma/\{e\})$ taking $(\E,D)$ to $\iota_*(\E,D)$ is surjective and closed (see \cite[Proposition 4.11]{APPLMS}). \par 

First of all, we assume that  $\iota_*(\emptyset,D)=\iota_*(\emptyset,D')$. This means that $D(v)=D'(v)$ for every vertex $v\in V(\Gamma)$ not incident to $e$.  If $e$ is a loop, then $D=D'$, and we have nothing to prove. Otherwise, let $s$ and $t$ be the end-vertices of $e$ and assume that $D(t)\ge D'(t)$. Set $n:=D(t)-D'(t)=D'(s)-D(s)$ and define the divisors $D_i$ on $\Gamma^{\{e\}}$ for $i=1,\dots,n$ and $D'_i$ on $\Gamma$ for $i=0,\dots,n$ taking a vertex $v$ to
\begin{align*}
& D_i(v)=
\begin{cases}
D(v) & \text{ if } v\not\in\{s,t\}\\
1    & \text{ if } v=v_e \\
D(v)-i = D'(v)+ n -i& \text{ if } v=t\\
D(v)+i-1 & \text{ if } v=s
\end{cases}
&
D'_i(v)=
\begin{cases}
D(v) & \text{ if } v\not\in\{s,t\}\\
D(v)-i =D'(v)+n-i& \text{ if } v=t\\
D(v)+i & \text{ if } v=s.
\end{cases}
\end{align*}
 Let $e_1:=e_2:=\dots:=e_n:=e$. Note that $(\emptyset,D_i')$ and $(\{e\},D_i)$ are $(v_0,\mu)$-quasistable because both $D_i(V)$ and $D'_i(V)$ are greater or equal than either $D(V)$ or $D'(V)$, for every $V\subset V(\Gamma)\subset V(\Gamma^{\{e\}})$.  We see that $(\emptyset,D)$ and $(\emptyset, D')$ are upper-connected in $\QS_{v_0,\mu}(\Gamma)$ by means of the edges $e_1,\dots,e_n$ and the divisors $D_1,\dots, D_n,D'_0,\dots,D'_n$. 

 Now we consider the general case. By the induction hypothesis, $\iota_*(\emptyset,D)$ and $\iota_*(\emptyset, D')$ are upper-connected in $\QS_{\iota(v_0),\iota_*(\mu)}(\Gamma/\{e\})$ by means of edges $e_1,\ldots, e_n$ of $\Gamma/\{e\}$, and divisors 
 $D_{e,i}\in \QS_{\iota(v_0),\iota_*(\mu)}(\Gamma/\{e\},\{e_i\})$ for $i=1,\dots,n$ and 
 $D'_{e,i}\in \QS_{\iota(v_0),\iota_*(\mu)}(\Gamma/\{e\})$ for $i=0,\dots,n$. Since $\iota_*$ is surjective, there are divisors $D_i\in \QS_{v_0,\mu}(\Gamma,\{e_i\})$ for $i=1,\dots,n$,  such that $\iota_*(\{e_i\},D_i)=(\{e_i\},D_{e,i})$. By Remark \ref{rem:closed} and the fact that $\iota_*$ is closed, we have that  there are $(v_0,\mu)$-quasistable divisors $D_i'$ and $D_i''$ on $\Gamma$ such that 
 \begin{align*}
 (\emptyset, D_i')\leq (\{e_i\},D_i), \;\;\;\;& (\emptyset, D_i'')\leq (\{e_i\},D_i),\\
 \iota_*(\emptyset, D_i')=(\emptyset, D_{e,i}'),\;\; \;\;& \iota_*(\emptyset, D_i'') = (\emptyset, D_{e,i-1}').
\end{align*}
 This means that $\iota_*(\emptyset, D)=\iota_*(\emptyset, D_1'')$, $\iota_*(\emptyset, D_i')=\iota_*(\emptyset, D_{i+1}'')$ and $\iota_*(\emptyset, D')=\iota_*(\emptyset, D_{n}')$. By the previous case, we have that the pairs $((\emptyset, D),(\emptyset, D_1''))$,  $((\emptyset, D_i''), (\emptyset, D_{i+1}''))$ and $((\emptyset, D'),(\emptyset,D_n'))$ are pairs of upper-connected pseudo-divisors in $\QS_{v_0,\mu}(\Gamma)$. Since $(\emptyset, D_i')\leq (\{e_i\},D_i)$ and $(\emptyset, D_i'')\leq (\{e_i\},D_i)$, it follows that $(\emptyset, D_i'')$ and $(\emptyset, D_i')$ are upper-connected in $\QS_{v_0,\mu}(\Gamma)$. This proves that $(\emptyset,D)$ and $(\emptyset, D')$ are upper-connected in $\QS_{v_0,\mu}(\Gamma)$, concluding the proof.
\end{proof}

Recall that $\mu_{\can}$ denotes the canonical polarization of degree $g-1$ (see Equation \eqref{eq:canonical-pol}). 
We will simply write $\QS_{v_0}(\Gamma)$ and $\QS_{v_0}(\Gamma,\E)$ instead of $\QS_{v_0,\mu_{\can}}(\Gamma)$ and $\QS_{v_0,\mu_{\can}}(\Gamma,\E)$.

\begin{Prop}\label{prop:one-QD}
 Let $\Gamma$ be a graph, and  $v_0,v_1$ be vertices of $\Gamma$. Then we have a canonical isomorphism of posets $\QS_{v_0}(\Gamma)\cong\QS_{v_1}(\Gamma)$.   
\end{Prop}

\begin{proof}
We construct a map $\QS_{v_0}(\Gamma)\to \QS_{v_1}(\Gamma)$ that takes a pseudo-divisor $(\E,D)$ in $\QS_{v_0}(\Gamma)$ to $(\E,D+v_0-v_1)$. This map is well-defined, indeed, fix $(\E,D)$ a $v_0$-quasistable pseudo-divisor, and set $D'=D+v_0-v_1$. For every $V\subset V(\Gamma^\E)$ we have that $\beta_{\Gamma,D}(V)$ is an integer number by Equation \eqref{eq:integer}. Moreover,
\[
\beta_{\Gamma,D'}(V)=\begin{cases}
\beta_{\Gamma,D}(V) + 1 & \text{ if }v_0\in V, v_1\notin V,\\
\beta_{\Gamma,D}(V) - 1 & \text{ if }v_0\notin V, v_1\in V,\\
\beta_{\Gamma,D}(V) & \text{ otherwise }.
\end{cases}
\]
It follows that $(\E, D-v_0+v_1)$ is $v_1$-quasistable.\par 
 The fact that the map is a morphism of posets is clear, and it has a natural inverse that takes $(\E,D')$ to $(\E, D'-v_0+v_1)$, hence it is an isomorphism.
\end{proof}

Notice that Proposition \ref{prop:one-QD} allows us to use the notation $\QS(\Gamma)$ to denote one of the posets $\QS_{v_0}(\Gamma)$, for $v_0\in V(\Gamma)$. Similarly, we will use  the notation   $\QS(\Gamma,\E)$ to denote one of the posets $\QS_{v_0}(\Gamma,\E)$. We will keep using the notations $\QS_{v_0}(\Gamma)$ and $\QS_{v_0}(\Gamma,\E)$ when we will need to consider one specific poset in the computations.

\section{Special posets}

In this section we will study some distinguished subposets of the poset of quasistable divisors $\QS(\Gamma)$.

\begin{Def}\label{def:fix}
We let $\PosFix$ (respectively, $\PosFixx$) be the ranked poset whose Hasse diagrams is drawn in Figure \ref{fig:hasse_posfix0} (respectively, in Figure \ref{fig:hasse_posfix1}). We write $\PosFix=\{\alpha,\beta,\gamma,\delta\}$ and $\PosFixx=\{\alpha_1,\beta_1,\beta_2,\beta_3,\beta_4, \gamma_1, \gamma_2, \gamma_3\}$.
\begin{figure}[h!]
\centering
\begin{minipage}[b][][b]{0.49\linewidth}
\begin{center}
\begin{tikzpicture}

\draw[fill] (0,2) circle (2pt) node[left] {$\alpha$};
\draw[fill] (1,2) circle (2pt) node[left] {$\beta$};
\draw[fill] (0,0) circle (2pt);
\draw[fill] (1,0) circle (2pt);

\draw[decoration={markings, mark=at position 0.6 with {\arrow{<}}}, postaction={decorate}] (0,2) -- (0,0) node[left] {$\gamma$};
\draw[decoration={markings, mark=at position 0.6 with {\arrow{<}}}, postaction={decorate}] (0,2) -- (1,0) node[left] {$\delta$};
\draw[decoration={markings, mark=at position 0.6 with {\arrow{<}}}, postaction={decorate}] (1,2) -- (0,0);
\draw[decoration={markings, mark=at position 0.6 with {\arrow{<}}}, postaction={decorate}] (1,2) -- (1,0);
\end{tikzpicture}
\caption{The Hasse diagram of  $\PosFix$.}
\label{fig:hasse_posfix0}
\end{center}
\end{minipage}
\begin{minipage}[b][][b]{0.49\linewidth}
\begin{center}
\begin{tikzpicture}
\draw[fill] (1.5,4) circle (2pt) node[above] {$\alpha_1$};
\draw[fill] (0,2) circle (2pt) node[left] {$\beta_1$};
\draw[fill] (1,2) circle (2pt) node[left] {$\beta_2$};
\draw[fill] (2,2) circle (2pt) node[left] {$\beta_3$};
\draw[fill] (3,2) circle (2pt) node[left] {$\beta_4$};
\draw[fill] (0.5,0) circle (2pt) node[left] {$\gamma_1$};
\draw[fill] (1.5,0) circle (2pt) node[left] {$\gamma_2$};
\draw[fill] (2.5,0) circle (2pt) node[left] {$\gamma_3$};
\draw[decoration={markings, mark=at position 0.6 with {\arrow{<}}}, postaction={decorate}] (1.5,4) -- (0,2);
\draw[decoration={markings, mark=at position 0.6 with {\arrow{<}}}, postaction={decorate}] (1.5,4) -- (1,2);
\draw[decoration={markings, mark=at position 0.6 with {\arrow{<}}}, postaction={decorate}] (1.5,4) -- (2,2);
\draw[decoration={markings, mark=at position 0.6 with {\arrow{<}}}, postaction={decorate}] (1.5,4) -- (3,2);
\draw[decoration={markings, mark=at position 0.6 with {\arrow{<}}}, postaction={decorate}] (0,2) -- (0.5,0);
\draw[decoration={markings, mark=at position 0.6 with {\arrow{<}}}, postaction={decorate}] (0,2) -- (1.5,0);
\draw[decoration={markings, mark=at position 0.6 with {\arrow{<}}}, postaction={decorate}] (1,2) -- (0.5,0);
\draw[decoration={markings, mark=at position 0.6 with {\arrow{<}}}, postaction={decorate}] (1,2) -- (1.5,0);
\draw[decoration={markings, mark=at position 0.6 with {\arrow{<}}}, postaction={decorate}] (2,2) -- (2.5,0);
\draw[decoration={markings, mark=at position 0.6 with {\arrow{<}}}, postaction={decorate}] (2,2) -- (1.5,0);
\draw[decoration={markings, mark=at position 0.6 with {\arrow{<}}}, postaction={decorate}] (3,2) -- (2.5,0);
\draw[decoration={markings, mark=at position 0.6 with {\arrow{<}}}, postaction={decorate}] (3,2) -- (1.5,0);
\end{tikzpicture}
\caption{The Hasse diagram of  $\PosFixx$.}
\label{fig:hasse_posfix1}
\end{center}
\end{minipage}
\end{figure}
\end{Def}

\begin{Prop}\label{prop:P0}
Let $\Gamma$ be a graph and $v_0$ a vertex of $\Gamma$.
Suppose that $g\col \PosFix \to \QS_{v_0}(\Gamma)$ is an injective morphism of posets that preserves cover relations. Then there are parallel edges $e_1,e_2$ of $\Gamma$ and a subset $\E\subset E(\Gamma)\setminus\{e_1,e_2\}$ such that, denoting by $s$ and $t$ the end-vertices of $e_1$ and $e_2$, one of the following conditions hold
\begin{enumerate}
    \item[(1)] there is a divisor $D$ on $\Gamma^\E$ such that 
    \[
g(\PosFix)=\left\{ \begin{array}{l}
\{(\E\cup\{e_1\},D+v_{e_1}),(\E\cup\{e_2\},D+v_{e_2}),\\ (\E,D+s), (\E,D+t)
\end{array}\right\}.
\]
    \item[(2)] there is a divisor $D$ on $\Gamma^\E$ such that 
    \[
g(\PosFix)=\left\{
\begin{array}{l}
(\E\cup\{e_1,e_2\},D-t+v_{e_1}+v_{e_2}),\; (\E\cup\{e_1,e_2\}, D-s+v_{e_1}+v_{e_2}),\\
(\E\cup \{e_1\}, D+v_{e_1}), \;
(\E\cup \{e_2\}, D+v_{e_2})
\end{array}\right\}.
\]
\end{enumerate}
The two possibilities for the Hasse diagram of $g(\mathbf{P})$ are drawn in Figure \ref{fig:gP0} (where we only draw the edges $e_1$ and $e_2$, instead of the whole graph $\Gamma$).
\end{Prop}
\begin{figure}[ht]
    \begin{center}
    \begin{minipage}{0.46\linewidth}
    \flushleft
    \begin{tikzpicture}[scale = 1]
    \begin{scope}
    \draw[fill] (0,0) circle (2pt) node[left] {$a$};
    \draw[fill] (2,0) circle (2pt) node[right] {$b$};
    \draw[fill] (1,0.288) circle (2pt) node[above] {$1$};
    
    \draw (0,0) to [out=30, in=150] (2,0);
    \draw (0,0) to [out=-30, in=-150] (2,0);
    \end{scope}
    \begin{scope}[shift={(4,0)}]
    \draw[fill] (0,0) circle (2pt) node[left] {$a$};
    \draw[fill] (2,0) circle (2pt) node[right] {$b$};
    \draw[fill] (1,-0.288) circle (2pt) node[above] {$1$};
    \draw (0,0) to [out=30, in=150] (2,0);
    \draw (0,0) to [out=-30, in=-150] (2,0);
    \end{scope}
    
    \begin{scope}[shift = {(0,-3)}]
    \draw[fill] (0,0) circle (2pt) node[left] {$a+1$};
    \draw[fill] (2,0) circle (2pt) node[right] {$b$};
    
    \draw (0,0) to [out=30, in=150] (2,0);
    \draw (0,0) to [out=-30, in=-150] (2,0);
    \end{scope}
    
    \begin{scope}[shift = {(4,-3)}]
    \draw[fill] (0,0) circle (2pt) node[left] {$a$};
    \draw[fill] (2,0) circle (2pt) node[right] {$b+1$};
    
    \draw (0,0) to [out=30, in=150] (2,0);
    \draw (0,0) to [out=-30, in=-150] (2,0);
    \end{scope}
    
    \draw[decoration={markings, mark=at position 0.6 with {\arrow{<}}}, postaction={decorate}] (1,-0.6) -- (1,-2.4);
    \draw[decoration={markings, mark=at position 0.6 with {\arrow{<}}}, postaction={decorate}] (5,-0.6) -- (5,-2.4);
    \draw[decoration={markings, mark=at position 0.6 with {\arrow{<}}}, postaction={decorate}] (1.2,-0.6) -- (4.8,-2.4);
    \draw[decoration={markings, mark=at position 0.6 with {\arrow{<}}}, postaction={decorate}] (4.8,-0.6) -- (1.2,-2.4);
    \end{tikzpicture}
    \end{minipage}
    \quad\quad
    \begin{minipage}{0.46\linewidth}
    \flushright
    \begin{tikzpicture}[scale = 1]
    \begin{scope}
    \draw[fill] (0,0) circle (2pt) node[left] {$a-1$};
    \draw[fill] (2,0) circle (2pt) node[right] {$b$};
    \draw[fill] (1,0.288) circle (2pt) node[above] {$1$};
    \draw[fill] (1,-0.288) circle (2pt) node[above] {$1$};
    
    \draw (0,0) to [out=30, in=150] (2,0);
    \draw (0,0) to [out=-30, in=-150] (2,0);
    \end{scope}
    \begin{scope}[shift={(4,0)}]
    \draw[fill] (0,0) circle (2pt) node[left] {$a$};
    \draw[fill] (2,0) circle (2pt) node[right] {$b-1$};
    \draw[fill] (1,0.288) circle (2pt) node[above] {$1$};
    \draw[fill] (1,-0.288) circle (2pt) node[above] {$1$};
    \draw (0,0) to [out=30, in=150] (2,0);
    \draw (0,0) to [out=-30, in=-150] (2,0);
    \end{scope}
    
    \begin{scope}[shift = {(0,-3)}]
    \draw[fill] (0,0) circle (2pt) node[left] {$a$};
    \draw[fill] (2,0) circle (2pt) node[right] {$b$};
    \draw[fill] (1,-0.288) circle (2pt) node[above] {$1$};
    \draw (0,0) to [out=30, in=150] (2,0);
    \draw (0,0) to [out=-30, in=-150] (2,0);
    \end{scope}
    
    \begin{scope}[shift = {(4,-3)}]
    \draw[fill] (0,0) circle (2pt) node[left] {$a$};
    \draw[fill] (2,0) circle (2pt) node[right] {$b$};
    \draw[fill] (1,0.288) circle (2pt) node[below] {$1$};
    \draw (0,0) to [out=30, in=150] (2,0);
    \draw (0,0) to [out=-30, in=-150] (2,0);
    \end{scope}
    
    \draw[decoration={markings, mark=at position 0.6 with {\arrow{<}}}, postaction={decorate}] (1,-0.6) -- (1,-2.4);
    \draw[decoration={markings, mark=at position 0.6 with {\arrow{<}}}, postaction={decorate}] (5,-0.6) -- (5,-2.4);
    \draw[decoration={markings, mark=at position 0.6 with {\arrow{<}}}, postaction={decorate}] (1.2,-0.6) -- (4.8,-2.4);
    \draw[decoration={markings, mark=at position 0.6 with {\arrow{<}}}, postaction={decorate}] (4.8,-0.6) -- (1.2,-2.4);
    \end{tikzpicture}
    \end{minipage}
    \end{center}
    \caption{The two possibility for the Hasse diagram of the poset $g(\PosFix)$.}
    \label{fig:gP0}   
\end{figure}
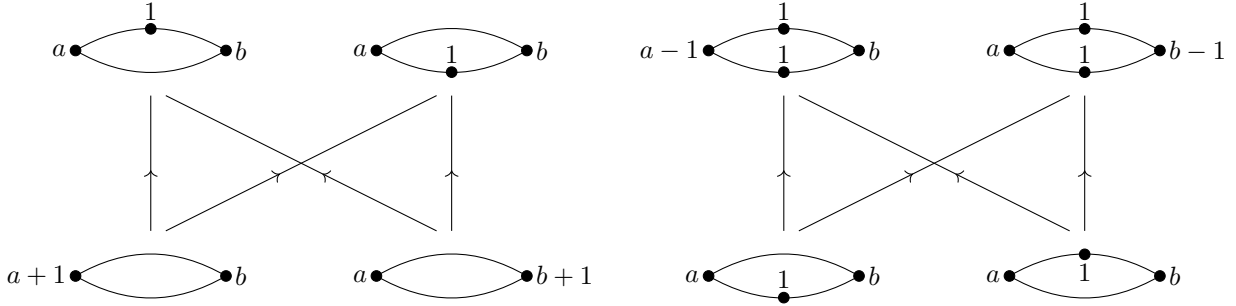


\begin{proof}
Recall that we write $\mathbf P=\{\alpha,\beta,\gamma,\delta\}$ (see Figure \ref{fig:hasse_posfix0}).
We set
\[
(\E_1,D_1):=g(\alpha), \;\;(\E_2,D_2):=g(\beta), \; \;(\E_3,D_3):=g(\gamma), \; \; (\E_4,D_4):=g(\delta).
\]
By definition of specialization, we have $\E_3\cup \E_4\subset \E_1\cap \E_2$, with $|\E_1|=|\E_2|=|\E_3|+1=|\E_4|+1$. Hence we have three cases:
\begin{enumerate}
    \item either $\E_3=\E_4$ and $\E_1\neq \E_2$,
    \item or $\E_3\neq \E_4$ and $\E_1=\E_2$, 
    \item or $\E_3=\E_4$ and $\E_1=\E_2$. 
\end{enumerate}

We begin with Case (1). In this case, we define $\E:=\E_3=\E_4$, which means that $\E_1=\E\cup \{e_1\}$ and $\E_2=\E\cup \{e_2\}$ for some distinct edges $e_1,e_2\in E(\Gamma)$. We have that $(\E,D_3)$ and $(\E,D_4)$ must be different (since $g$ is injective) and hence 
they are the two pseudo-divisors on $\Gamma$ of type $(\E,D')$ to which both $(\E\cup\{e_1\},D_1)$ and $(\E\cup \{e_2\}, D_2)$ specialize described in Remark \ref{rem:elementary}.
In particular neither $e_1$ nor $e_2$ is a loop, otherwise there will be only one of these specializations. 

Let us prove that $e_1$ and $e_2$ are parallel edges. Assume, by contradiction, that there exists a vertex $v$ incident to $e_1$ and not to $e_2$. Then, by Remark \ref{rem:elementary}, it follows that $D_3(v)=D_2(v)$ and $D_4(v)=D_2(v)$, and also, without loss of generality, that $D_3(v)=D_1(v)$ and $D_4(v)=D_1(v)+1$, giving rise to a contradiction. This proves that $e_1$ and $e_2$ are incident to same pair of vertices, meaning that they are parallel. 

Denote by $s,t$ the end-vertices of $e_1$ and $e_2$. Again by Remark \ref{rem:elementary} and up to switch $D_3$ with $D_4$, we have that either $D_3=D_1-v_{e_1}+s=D_2-v_{e_2}+s$ or $D_3=D_1-v_{e_1}+t=D_2-v_{e_2}+s$. We can rule out the second possibility as follows. If $D_3=D_1-v_{e_1}+t=D_2-v_{e_2}+s$, then $D_4=D_1-v_{e_1}+s=D_2-v_{e_2}+t$, hence $D_3(t)=D_1(t)+1=D_2(t)$ and $D_4(t)=D_1(t)=D_2(t)+1$, which is a contradiction. It follows that $D_3=D_1-v_{e_1}+s=D_2-v_{e_2}+s$, giving the poset described in item (1) of the statement with $D:=D_1-v_{e_1}=D_2-v_{e_2}$. 

We move to Case (2). In this case, we define $\E:=\E_3\cap \E_4$, and hence $|\E_3|=|\E_4|=|\E|+1$. This means that $\E_3=\E\cup \{e_1\}$, $\E_4=\E\cup \{e_2\}$, $\E_1=\E_2=\E\cup\{e_1,e_2\}$ for some distinct edges $e_1,e_2\in E(\Gamma)$. 

Let us prove that $e_1$ and $e_2$ are parallel edges. Let $V_0$ be the set of vertices incident to both $e_1$ and $e_2$. 
Assume, by contradiction, that $|V_0|\le 1$. Let $v$ be a vertex not incident to $e_2$. Since $(\E\cup \{e_1\},D_3)$ is an elementary specialization of both $(\E\cup\{e_1,e_2\},D_1)$ and $(\E\cup\{e_1,e_2\},D_2)$, by  Remark \ref{rem:elementary} we have that $D_1(v)=D_3(v)=D_2(v)$. We can  argue similarly for any vertex not incident to $e_1$. We deduce that $D_1(v)=D_2(v)$ for every vertex $v\not\in V_0$. Since $D_1$ and $D_2$ have the same degree and since $|V_0|\le 1$, we have that $D_1=D_2$, which is a contradiction. This proves that $|V_0|=2$, i.e., $e_1$ and $e_2$ are parallel edges. 

Let $s$ and $t$ be the end-vertices of $e_1$ and $e_2$. Then, for $i=3,4$, we have four cases
\begin{enumerate}
    \item[(i)] either $D_i(s)=D_1(s)+1 = D_2(s)+1$,
    \item[(ii)] or $D_i(s) = D_1(s) = D_2(s)+1$,
    \item[(iii)] or $D_i(s) = D_1(s)+1 = D_2(s)$,
    \item[(iv)] or $D_i(s) = D_1(s) = D_2(s)$. 
\end{enumerate}
With the same argument used above, Cases (i) and (iv) would imply that $D_1=D_2$, which is a contradiction. In Case (ii) we have that $D_i=D_1-v_{e_{5-i}}+t=D_2-v_{e_{5-i}}+s$, which means that $D_1+t=D_2+s$. Similarly, in Case (iii) we have that $D_1+s=D_2+t$. So the same case must hold for both $i=3$ and $i=4$.
This means that $D_3(v)=D_4(v)$ for every $v\in V(\Gamma)$, giving the poset described in item (2) of the statement  with $D:=D_3-v_{e_1}=D_4-v_{e_2}$.

Finally, we consider Case (3). In this case, we define $\E:=\E_3=\E_4$, which means that $\E_1=\E_2=\E\cup\{e\}$ for some edge $e\in E(\Gamma)$. Since $g$ is injective, we have that $(\E,D_3)$ and $(\E,D_4)$ are different. Hence they are the two pseudo-divisor of type $(\E,D')$ to which both $(\E\cup\{e\},D_1)$ and $(\E\cup\{e\},D_2)$ specialize, described in Remark \ref{rem:elementary}. This implies that $D_1=D_2$, which is a contradiction with the fact that $g(\alpha)\neq g(\beta)$ .
\end{proof}

\begin{Cor}\label{cor:gradedP0}
Let $\Gamma$ be a graph and $v_0$ a vertex of $\Gamma$. Let $g\col \PosFix\to \QS_{v_0}(\Gamma)$ be an injective morphism of ranked posets. Then there are parallel edges $e_1,e_2\in E(\Gamma)$ and a divisor $D$ on $\Gamma$ such that 
     \[
g(\PosFix)=\{
(\{e_1\},D+v_{e_1}),(\{e_2\},D+v_{e_2}), (\emptyset,D+s), (\emptyset,D+t)
\},
\]
where $s$ and $t$ are the end-vertices of $e_1$ and $e_2$.
\end{Cor}

\begin{proof}
The rank of the elements $\alpha$ and $\beta$ in $\PosFix$ is 1 and $g$ is a morphism of ranked posets. Then $g(\alpha)$ and $g(\beta)$ have rank 1 in $\QS(\Gamma)$, and hence $g(\PosFix)$ is the poset described in item (1) of Proposition \ref{prop:P0} with $\E=\emptyset$.
\end{proof}

\begin{Def}
Let $\Gamma$ be a graph and $v_0$ be a vertex of $\Gamma$. 
Assume that $e_1$ and $e_2$ are parallel edges of $\Gamma$ and $D$ is a divisor in $\QS_{v_0}(\Gamma,\{e_1,e_2\})$. We denote by $s$ and $t$ the end-vertices of $e_1$ and $e_2$. We let $\mathbf{R}_{e_1,e_2}(D)$ be the ranked sub-poset of $\QS_{v_0}(\Gamma)$ given by
\[
\mathbf{R}_{e_1,e_2}(D)=\left\{ 
\begin{array}{l}
(\{e_1,e_2\},D), \; (\{e_1\},D-v_{e_2}+s), \; (\{e_1\},D-v_{e_2}+t), \\
 (\{e_2\},D-v_{e_1}+s),  (\{e_2\},D-v_{e_1}+t), \\
 (\emptyset,D-v_{e_1}-v_{e_2}+2s),  (\emptyset,D-v_{e_1}-v_{e_2}+s+t), (\emptyset,D-v_{e_1}-v_{e_2}+2s)
 \end{array}
\right\}.
\]

The Hasse diagram of $\mathbf{R}_{e_1,e_2}(D)$ is drawn in Figure \ref{fig:gP1}. In the figure, we only draw the edges $e_1$ and $e_2$, instead of the whole graph $\Gamma$. 
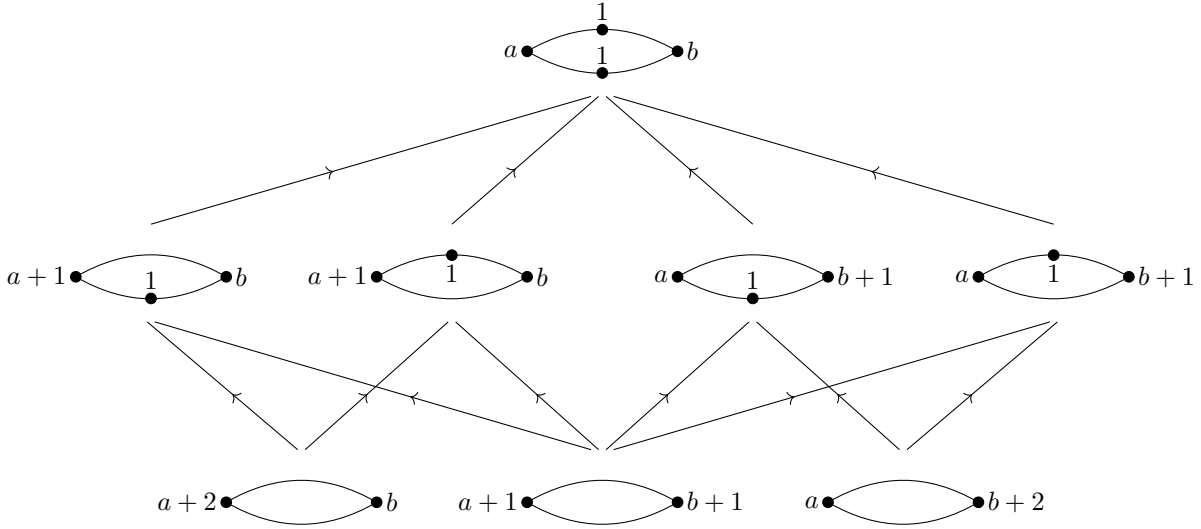
\begin{figure}[h]
    \centering
    \begin{tikzpicture}
     \begin{scope}
    \draw[fill] (0,0) circle (2pt) node[left] {$a$};
    \draw[fill] (2,0) circle (2pt) node[right] {$b$};
    \draw[fill] (1,0.288) circle (2pt) node[above] {$1$};
    \draw[fill] (1,-0.288) circle (2pt) node[above] {$1$};
    
    \draw (0,0) to [out=30, in=150] (2,0);
    \draw (0,0) to [out=-30, in=-150] (2,0);
    \end{scope}

    \begin{scope}[shift = {(-6,-3)}]
    \draw[fill] (0,0) circle (2pt) node[left] {$a+1$};
    \draw[fill] (2,0) circle (2pt) node[right] {$b$};
    \draw[fill] (1,-0.288) circle (2pt) node[above] {$1$};
    \draw (0,0) to [out=30, in=150] (2,0);
    \draw (0,0) to [out=-30, in=-150] (2,0);
    \end{scope}
    
    \begin{scope}[shift = {(-2,-3)}]
    \draw[fill] (0,0) circle (2pt) node[left] {$a+1$};
    \draw[fill] (2,0) circle (2pt) node[right] {$b$};
    \draw[fill] (1,0.288) circle (2pt) node[below] {$1$};
    \draw (0,0) to [out=30, in=150] (2,0);
    \draw (0,0) to [out=-30, in=-150] (2,0);
    \end{scope}
    
    \begin{scope}[shift = {(2,-3)}]
    \draw[fill] (0,0) circle (2pt) node[left] {$a$};
    \draw[fill] (2,0) circle (2pt) node[right] {$b+1$};
    \draw[fill] (1,-0.288) circle (2pt) node[above] {$1$};
    \draw (0,0) to [out=30, in=150] (2,0);
    \draw (0,0) to [out=-30, in=-150] (2,0);
    \end{scope}
    
    \begin{scope}[shift = {(6,-3)}]
    \draw[fill] (0,0) circle (2pt) node[left] {$a$};
    \draw[fill] (2,0) circle (2pt) node[right] {$b+1$};
    \draw[fill] (1,0.288) circle (2pt) node[below] {$1$};
    \draw (0,0) to [out=30, in=150] (2,0);
    \draw (0,0) to [out=-30, in=-150] (2,0);
    \end{scope}
    
    \begin{scope}[shift = {(-4,-6)}]
    \draw[fill] (0,0) circle (2pt) node[left] {$a+2$};
    \draw[fill] (2,0) circle (2pt) node[right] {$b$};
    
    \draw (0,0) to [out=30, in=150] (2,0);
    \draw (0,0) to [out=-30, in=-150] (2,0);
    \end{scope}
    
    \begin{scope}[shift = {(0,-6)}]
    \draw[fill] (0,0) circle (2pt) node[left] {$a+1$};
    \draw[fill] (2,0) circle (2pt) node[right] {$b+1$};
    
    \draw (0,0) to [out=30, in=150] (2,0);
    \draw (0,0) to [out=-30, in=-150] (2,0);
    \end{scope}
    
    \begin{scope}[shift = {(4,-6)}]
    \draw[fill] (0,0) circle (2pt) node[left] {$a$};
    \draw[fill] (2,0) circle (2pt) node[right] {$b+2$};
    
    \draw (0,0) to [out=30, in=150] (2,0);
    \draw (0,0) to [out=-30, in=-150] (2,0);
    \end{scope}
    
    \draw[decoration={markings, mark=at position 0.6 with {\arrow{<}}}, postaction={decorate}] (0.85,-0.6) -- (-5,-2.3);
    \draw[decoration={markings, mark=at position 0.6 with {\arrow{<}}}, postaction={decorate}] (0.95,-0.6) -- (-1,-2.3);
    \draw[decoration={markings, mark=at position 0.6 with {\arrow{<}}}, postaction={decorate}] (1.05,-0.6) -- (3,-2.3);
    \draw[decoration={markings, mark=at position 0.6 with {\arrow{<}}}, postaction={decorate}] (1.15,-0.6) -- (7,-2.3);
    \draw[decoration={markings, mark=at position 0.6 with {\arrow{<}}}, postaction={decorate}] (-5.05, -3.6) -- (-3.05, -5.3);
    \draw[decoration={markings, mark=at position 0.6 with {\arrow{<}}}, postaction={decorate}] (-4.95, -3.6) -- (0.85, -5.3);
    \draw[decoration={markings, mark=at position 0.6 with {\arrow{<}}}, postaction={decorate}] (-1.05, -3.6) -- (-2.95, -5.3);
    \draw[decoration={markings, mark=at position 0.6 with {\arrow{<}}}, postaction={decorate}] (-0.95, -3.6) -- (0.95, -5.3);
    \draw[decoration={markings, mark=at position 0.6 with {\arrow{<}}}, postaction={decorate}] (2.95, -3.6) -- (1.05, -5.3);
    \draw[decoration={markings, mark=at position 0.6 with {\arrow{<}}}, postaction={decorate}] (3.05, -3.6) -- (4.95, -5.3);
    \draw[decoration={markings, mark=at position 0.6 with {\arrow{<}}}, postaction={decorate}] (6.95, -3.6) -- (1.15, -5.3);
    \draw[decoration={markings, mark=at position 0.6 with {\arrow{<}}}, postaction={decorate}] (7.05, -3.6) -- (5.05, -5.3);
    \end{tikzpicture}
    \caption{The Hasse diagram of the poset $\mathbf{R}_{e_1,e_2}(D)$}
     \label{fig:gP1}
\end{figure}
\end{Def}

Recall that we write $\PosFixx=\{\alpha_1,\beta_1,\beta_2,\beta_3,\beta_4, \gamma_1, \gamma_2, \gamma_3,\gamma_4\}$ (see  Definition \ref{def:fix}). 
Notice that $\mathbf R$ and  $\mathbf{R}_{e_1,e_2}(D)$ are isomorphic ranked posets.

\begin{Prop}
\label{prop:image_parallel}
Let $\Gamma$ be a graph and $v_0$ a vertex of $\Gamma$.
Suppose that $g\col \PosFixx \to \QS_{v_0}(\Gamma)$ is an injective morphism of ranked posets. Then $g(\alpha_1)=(\{e_1,e_2\}, D)$, for some parallel edges $e_1,e_2$ of $\Gamma$ and a divisor $D\in \QS_{v_0}(\Gamma,\{e_1,e_2\})$, and $g(\PosFixx)=\mathbf{R}_{e_1,e_2}(D)$.
\end{Prop}

\begin{proof}
Since $g$ is a morphism of ranked posets and the rank of $\alpha_1$ is 2, we have $g(\alpha_1)=(\{e_1,e_2\},D)$ for some edges $e_1$ and $e_2$ of $\Gamma$ and a divisor $D$ on $\QS_{v_0}(\Gamma,\{e_1,e_2\})$. Let $s_1$ and $t_1$ (respectively, $s_2$ and $t_2$) be the (possibly coincident) end-vertices of $e_1$ (respectively, of $e_2$). 
By Remark \ref{rem:elementary}, there are at most 4 pseudo-divisors $(\E',D')$ of rank 1 (i.e., with $|\E'|=1$) such that $(\E',D')<(\{e_1,e_2\},D)$: they are the pseudo-divisors of the set
\begin{equation}\label{eq:four}
\left\{
(\{e_1\}, D-v_{e_2}+s_2), \;
(\{e_1\}, D-v_{e_2}+t_2), \; 
(\{e_2\}, D-v_{e_1}+s_1), \;
(\{e_2\}, D-v_{e_1}+t_1)\right\}.
\end{equation}
Since $g$ is an injective morphism of ranked poset,  the set in Equation \eqref{eq:four} is equal to $\{g(\beta_1), g(\beta_2), g(\beta_3), g(\beta_4)\}$. In particular neither $e_1$  nor $e_2$ are loops, i.e., $s_1\neq t_1$ and $s_2\neq t_2$.\par 

The induced subposets 
$\{\beta_1,\beta_2,\gamma_1,\gamma_2\}$ and $\{\beta_3,\beta_4,\gamma_2,\gamma_3\}$ of $\PosFixx$ are isomorphic to the poset $\PosFix$. By Corollary \ref{cor:gradedP0}, we see that $e_1$ and $e_2$ are parallel edges of $\Gamma$ and, without loss of generality, we have that $s:=s_1=s_2$, $t:=t_1=t_2$, and
\begin{align*}
    g(\beta_1)=(\{e_1\}, D-v_{e_2}+s), \; \; &
    g(\beta_2)=(\{e_2\}, D-v_{e_1}+s),\\
 g(\beta_3)=(\{e_1\}, D-v_{e_2}+t), \; \; &
g(\beta_4)=(\{e_2\}, D-v_{e_1}+t).
\end{align*}

Finally, there are exactly 3 pseudo-divisors of rank 0, i.e., of type  $(\emptyset, D'')$, that are smaller than at least one pseudo-divisor in the set $\mathcal U=\{g(\beta_1),g(\beta_2),g(\beta_3),g(\beta_4)\}$. By Remark \ref{rem:elementary}, they are 
\begin{align*}
(\emptyset, D''_1) := &(\emptyset, D-v_{e_1}-v_{e_2}+2s),\\
(\emptyset, D''_2) := &(\emptyset, D-v_{e_1}-v_{e_2}+s+t),\\
(\emptyset, D''_3) := & (\emptyset, D-v_{e_1}-v_{e_2}+2t).
\end{align*}
The first and the third are smaller than exactly two of the pseudo-divisors in the set $\mathcal U$, while the second is smaller that every pseudo-divisor in $\mathcal U$. Thus we have
$g(\gamma_1)=(\emptyset, D''_1)$, $g(\gamma_2)=(\emptyset, D''_2)$, and $g(\gamma_3)=(\emptyset, D''_3)$.
This finishes the proof.
\end{proof}

\begin{Rem}
In the proofs of Propositions \ref{prop:P0} and \ref{prop:image_parallel}, we never used that the divisors were quasistable. So these results remain true if we change the target of the map $g$ with the poset of all pseudo-divisors on $\Gamma$.
\end{Rem}

\begin{Lem}
\label{lem:parallel_edges}
Let $\Gamma$ be a graph and $(\E,D)$ a pseudo-divisor on $\Gamma$. Assume that $e\in \E$ is a non-loop edge of $\Gamma$. Let $s$ and $t$ be the end-vertices of $e$. 
If $e_0\in \E$ is an edge of $\Gamma$ such that there exists a pseudo-divisor $(\E\setminus\{e,e_0\},D')$ on $\Gamma$ smaller than both $(\E\setminus\{e\},D-v_e+s)$ and $(\E\setminus\{e\},D-v_e+t)$, then $e$ and $e_0$ are parallel edges of $\Gamma$.
\end{Lem}

\begin{proof}
Set $D_1:=D-v_e+s$ and $D_2:=D-v_e+t$ and recall Remark \ref{rem:elementary}. For $v\in V(\Gamma)$, we have 
\begin{align*}
&D_1(v)=\begin{cases}
D(v)&¨\text{ if }v\neq s,\\
D(v) + 1&¨\text{ if }v=s,
\end{cases}
&
D_2(v)=
\begin{cases}
D(v)&¨\text{ if }v\neq t,\\
D(v) + 1&¨\text{ if }v=t.
\end{cases}
\end{align*}
Let $s_0$ and $t_0$ be the end-vertices of $e_0$. We can assume without loss of generality that
\[
D'(v)=\begin{cases}
D_1(v)&¨\text{ if }v\neq s_0,\\
D_1(v) + 1&¨\text{ if }v= s_0.
\end{cases}
\]
Assume by contradiction that 
\[
D'(v)=\begin{cases}
D_2(v)&¨\text{ if }v\neq s_0,\\
D_2(v) + 1&¨\text{ if }v= s_0.
\end{cases}
\]
Hence we would have that $D_1=D_2$, a contradiction. 
Then we have that $s_0\neq t_0$ and
\[
D'(v)=\begin{cases}
D_2(v)&¨\text{ if }v\neq t_0,\\
D_2(v)+1&¨\text{ if }v= t_0.
\end{cases}
\]
If $t_0\not\in \{s,t\}$, then $D'(t_0)=D_2(t_0)+1 = D(t_0)+1$ and $D'(t_0)=D_1(t_0)=D(t_0)$ which is a contradiction. So we have that $t_0\in\{s,t\}$ and, analogously, we have that $s_0\in\{s,t\}$. This proves that $\{s_0,t_0\}=\{s,t\}$ and hence the edges $e,e_0$ are parallel edges of $\Gamma$.
\end{proof}

\section{Torelli theorem for graphs}

In this section we will prove the following Torelli theorem for graphs:

\begin{Thm}\label{thm:main1}
Let $\Gamma$ and $\Gamma'$ be graphs. The posets $\QS(\Gamma)$ and $\QS(\Gamma')$ are isomorphic if and only if there is a bijection between the biconnected components of $\Gamma/\bridges(\Gamma)$ and $\Gamma'/\bridges(\Gamma')$ such that the corresponding components are isomorphic as pure graphs.
\end{Thm}

As a particular case of Theorem \ref{thm:main1}, we get the following corollary.

\begin{Cor}\label{cor:main-pure-biconnected}
Let $\Gamma$ and $\Gamma'$ be biconnected pure graphs. The posets $\QS(\Gamma)$ and $\QS(\Gamma')$ are isomorphic if and only if $\Gamma$ and $\Gamma'$ are isomorphic.
\end{Cor}

We start by reducing to the case of pure graphs.

\begin{Prop}\label{prop:pure-reduction}
Let $\Gamma$ be a graph. Let $\Gamma_0$ be the pure graph with underlying graph equal to $\Gamma$. Then $\QS(\Gamma)$ is naturally isomorphic to $\QS(\Gamma_0)$.
\end{Prop} 

\begin{proof}
It is enough to notice that a pseudo-divisor $(\E,D)$ on $\Gamma_0$ is $v_0$-quasistable  for some $v_0\in V(\Gamma)$, if and only if $(\E,D+\sum_{v\in V(\Gamma)}w_\Gamma(v)v)$ is a $v_0$-quasistable pseudo-divisor on $\Gamma$.
\end{proof}

By Proposition \ref{prop:pure-reduction}, we see that it is enough to prove Theorem \ref{thm:main1} for pure graphs. For the rest of this section, we will only consider pure graphs, and we use the word \emph{graph} for \emph{pure graphs}. 

\begin{Def}
A \emph{special pair} of a graph $\Gamma$ is a set $\{e_1,e_2\}$ of edges of $\Gamma$ such that
\begin{enumerate}
    \item[(1)] the edges $e_1,e_2$ are distinct parallel edges of $\Gamma$;
    \item[(2)] there are no parallel edges to $e_1$ and $e_2$ in $E(\Gamma)\setminus\{e_1,e_2\}$;
    \item[(3)] the graph $\Gamma$ remains connected after the removal of $e_1$ and $e_2$.
\end{enumerate} 
\end{Def}

 Condition (3) implies that a special pair of $\Gamma$ is contained in $\ND(\Gamma)$ (recall Equation \eqref{eq:ND}). From now on, we will fix:
\begin{enumerate}
    \item[(1)] two graphs $\Gamma$ and $\Gamma'$.
    \item[(2)] an isomorphism of posets $f\col \QS(\Gamma)\to \QS(\Gamma')$ with inverse $f^{-1}\col \QS(\Gamma')\to \QS(\Gamma)$. 
    \item[(3)] identifications $\QS(\Gamma)=\QS_{v_0}(\Gamma)$ and $\QS(\Gamma')=\QS_{v'_0}(\Gamma')$, for some $v_0\in V(\Gamma)$ and $v'_0\in V(\Gamma')$.
\end{enumerate}

Let $e_1,e_2$ be parallel edges of the graph $\Gamma$. Let $s$ and $t$ be the end-vertices of $e_1$ and $e_2$. Assume that there is a divisor $D\in \QS_{v_0}(\Gamma,\{e_1,e_2\})$ (the existence of one such $D$ is equivalent to the fact that $\Gamma$ remains connected after the removal of $e_1$ and $e_2$). By Proposition \ref{prop:image_parallel}, there are parallel edges $e_1',e_2'$ of the graph $\Gamma'$ and a divisor $D'\in\QS_{v'_0}(\Gamma',\{e_1',e_2'\})$ such that $f(\mathbf{R}_{e_1,e_2}(D)) = \mathbf{R}_{e_1',e_2'}(D')$. Let $s',t'$ be the end-vertices of  $e'_1$ and $e'_2$.

\begin{Lem}
\label{lem:P1_special_pair}
Keep the above notations. Assume that
\[
f(\{e_1\},D-v_{e_2}+s)=(\{e_1'\},D'-v_{e'_2}+s'),
\]
\[
f(\{e_1\},D-v_{e_2}+t))=(\{e_2'\},D'-v_{e'_1}+t').
\]
Then $\{e_1,e_2\}$ and $\{e_1',e_2'\}$ are special pairs of $\Gamma$ and $\Gamma'$, respectively.
\end{Lem}

\begin{proof} 
Let $e\neq e_1$ be an edge of $\Gamma$ parallel to $e_1$ and $e_2$. Using that $D\in \QS_{v_0}(\Gamma,\{e_1,e_2\})$, it is easy to see that $D-v_{e_2}+v_e$ is in $\QS_{v_0}(\Gamma, \{e,e_1\})$. We also have
\begin{align*}
(\{e,e_1\},D-v_{e_2}+v_e)&\geq (\{e_1\},D-v_{e_2}+s),\\
(\{e,e_1\},D-v_{e_2}+v_e)&\geq(\{e_1\},D - v_{e_2}+t).
\end{align*}
That means that the number of quasistable pseudo-divisors $(\E,\widetilde{D})$ on $\Gamma$ with $|\E|=2$ such that 
\[
(\E,\widetilde{D})\geq (\{e_1\},D-v_{e_2}+s)
\;\;\text{ and } \;\;
(\E,\widetilde{D}) \geq  (\{e_1\},D-v_{e_2}+t)
\]  
is at least the number of edges parallel to $e_1$ and different from $e_1$.

On the other hand, let us see that the number of pseudo-divisors $(\E',\widetilde{D}')$ on $\Gamma'$ with $|\E'|=2$ such that 
\[
(\E',\widetilde{D}')\geq (\{e_1'\},D'-v_{e_2'}+s')
\;\text{ and } \;
(\E',\widetilde{D}')\geq (\{e_1'\},D' -v_{e_2'}+t')
\]
is exactly one. Indeed, we have that 
$(\{e_1',e_2'\},D')$ satisfy this condition. Moreover, any other such pseudo-divisor must satisfy that $\E'=\{e_1',e_2'\}$. 

Assume that we have another such pseudo-divisor $(\{e_1',e_2'\}, \widetilde{D}')$. 
The poset 
\[
\{
(\{e_1',e_2'\}, \widetilde{D}'),\; (\{e_1',e_2'\}, D'),
( \{e_1'\}, D'-v_{e_2'}+ s'), \;
(\{e_2'\}, D'-v_{e_1'}+t')
\}
\] 
is a copy of the poset 
$\PosFix$
inside $\QS_{v_0'}(\Gamma')$. 
 Proposition \ref{prop:P0} characterizes such copies of $\PosFix$ inside $\QS_{v'_0}(\Gamma')$, and it is clear that we are in the case of Item (2).  Hence $D'-v_{e_2'}+s'-v_{e_1'}=D'-v_{e_1'}+t'-v_{e_2'}$, which is a contradiction since $s'\ne t'$.\par 

Since $f$ is an isomorphism of posets, we deduce  that there exists exactly one edge parallel to $e_1$, which must be $e_2$. Since $(\{e_1,e_2\},D)$ is a quasistable pseudodivisor, by Remark \ref{rem:AP_46}, we have that the removal of $e_1,e_2$ does not disconnect the graph, hence $e_1,e_2$ is a special pair of $\Gamma$. Arguing similarly for $f^{-1}$, we have that $e_1',e_2'$ is a special pair of $\Gamma'$.
\end{proof}

Recall the functions $\epsilon_\Gamma$ and $\delta_\Gamma$ defined in Equation \eqref{eq:eps-delta}.

\begin{Prop}\label{prop:alternativa}
Let $e$ be an edge of $\Gamma$. Assume that there are divisors $D_1$ and $D_2$ in $\QS_{v_0}(\Gamma,\{e\})$. Set $\{e_1'\}:=\epsilon_{\Gamma'}(f(\{e\},D_1))$ and $\{e_2'\}:=\epsilon_{\Gamma'}(f(\{e\},D_2))$. Then one of the two conditions holds.
\begin{enumerate}
    \item The edges $e_1'$ and $e_2'$ of $\Gamma'$ are equal.
    \item The edge $e$ belongs to a special pair $\{e,e_0\}$ of $\Gamma$ and $\{e_1',e_2'\}$ is a special pair of $\Gamma'$. 
\end{enumerate}
\end{Prop}
\begin{proof}
The result is clear if $D_1=D_2$, so we can assume that $D_1\neq D_2$. By Proposition \ref{prop:uppper-connected}, it is sufficient to prove the result when there exists a pseudo-divisor $(\{e,e_0\}, D)$ that specializes to both $(\{e\},D_1)$ and $(\{e\}, D_2)$. By Remark \ref{rem:elementary}, the edge $e$ is not a loop and, denoting by $s$ and $t$ the end-vertices of $e$, we have $D_1=D-v_e+s$ and $D_2=D-v_e+t$. Set $D_i':=\delta_{\Gamma'}(f(\{e\},D_i))$ for $i=1,2$.\par 

Assume that $e_1'\neq e_2'$. For $i=1,2$, we have that 
\[
f(\{e,e_0\},D)\geq f(\{e\},D_i)=(\{e_i'\},D_i'),
\]
which implies that $f(\{e,e_0\},D)=(\{e_1',e_2'\},D')$, for some divisor $D'\in \QS_{v'_0}(\Gamma',\{e'_1,e'_2\})$. \par

Again by Remark \ref{rem:elementary}, there is an end-vertex $s_i'$ of $e_i'$ for $i=1,2$, such that $D_1'=D'-v_{e_2'}+s_2'$ and $D_2'=D'-v_{e_1'}+s_1'$. If we set $\wt{D}' := D'-v_{e_1'}-v_{e_2'}+s_1'+s_2'$, then  $(\emptyset, \wt{D}')\leq (\{e_i'\},D_i')$ for $i=1,2$. Set $(\emptyset,\wt{D}):=f^{-1}(\emptyset, \wt{D}')$. Therefore, for $i=1,2$, we have
\[
(\emptyset, \wt{D})\leq (\{e\},D_i)\leq (\{e,e_0\},D).
\]
We see that $(\{e,e_0\},D) $ satisfies the hypotheses of Lemma \ref{lem:parallel_edges}, hence the edges $e$ and $e_0$ of $\Gamma$ are parallel. \par 

Now consider the image of the poset $\mathbf R_{e,e_0}(D)\subset \QS_{v_0}(\Gamma)$ via the isomorphism $f$. By Proposition \ref{prop:image_parallel}, we have $f(\mathbf R_{e,e_0}(D))=\mathbf R_{e_1',e_2'}(D')$ and the edges $e'_1,e'_2$ are parallel, with 
\[
f(\{e\},D-v_{e_0}+s)=(\{e_1'\},D'-v_{e_2'}+s_2')
\]
\[
f(\{e\},D-v_{e_0}+t)=(\{e_2'\},D'-v_{e_1'}+s_1').
\]

By contradiction, assume that $s':=s_1'=s_2'$, and let $t'$ be the other end-vertex of $e'_1$ and $e'_2$. Then the two pseudo-divisors $(\emptyset, D'-v_{e_1'}-v_{e_2'}+2s')$ and $(\emptyset, D'-v_{e_1'}-v_{e_2'}+s'+t')$ of $\mathbf{R}_{e_1',e_2'}(D')$ would be smaller then both $(\{e_1'\},D_1')$ and $(\{e_2'\},D_2')$. 
On the other hand, there is only one element of $\mathbf{R}_{e,e_0}(D)$ smaller then $(\{e\},D_1)$ and $(\{e\},D_2)$, which is $(\emptyset, D-v_{e}-v_{e_0}+s+t)$. We have a contradiction, which
 proves that $s_1'\neq s_2'$. So  the end-vertices of $e'_1$ and $e'_2$ are $s':=s'_2$ and $t':=s'_1$. 
 We see that the hypotheses of Lemma \ref{lem:P1_special_pair} are satisfied, hence the pairs $e,e_0$ and $e_1',e_2'$ are special pairs of $\Gamma$ and $\Gamma'$, respectively, and we are done.
\end{proof}

\begin{Cor}
\label{cor:special}
Let $e$ be an edge of $\Gamma$. The following conditions hold.
\begin{enumerate}
    \item If $e$ does not belong to a special pair of $\Gamma$, then $\epsilon_{\Gamma'}(f(\{e\},D))$ is independent of the choice of the divisor $D\in \QS_{v_0}(\Gamma,\{e\})$.
    \item If $e$ belongs to a special pair $\{e,e_0\}$ of $\Gamma$, then  $\epsilon_{\Gamma'}(f(\{e,e_0\},D))$ is independent of the choice of the divisor $D\in \QS_{v_0}(\Gamma,\{e,e_0\})$.
\end{enumerate}
\end{Cor}

\begin{proof}
The result readily follows from Proposition \ref{prop:alternativa}.
\end{proof}

\begin{Def}
\label{def:sim}
Let $\Gamma$ be a graph. We say that two subsets $\E_1$ and $\E_2$ of $E(\Gamma)$ are \emph{equivalent} if there are special pairs $\{e_{1,1}, e_{1,2}\},\dots,\{e_{k,1},e_{k,2}\}$ of $\Gamma$ such that  $\{e_{1,i},\ldots, e_{k,i}\}\subset \E_i$ for $i=1,2$ and
\[
    \E_1\setminus \{e_{1,1},\dots,e_{k,1}\}=
    \E_2\setminus \{e_{1,2},\dots,e_{k,2}\}.
    \]
We say that two pseudo-divisors $(\E_1,D_1)$ and $(\E_2,D_2)$ of $\Gamma$ are \emph{equivalent}, and we write $(\E_1,D_1)\sim (\E_2,D_2)$, if the following conditions hold
\begin{enumerate}
    \item[(1)] we have $D_1(v)=D_2(v)$ for every $v\in V(\Gamma)$.
    \item[(2)] the subsets $\E_1$ and $\E_2$ of $E(\Gamma)$ are equivalent.
\end{enumerate} 
\end{Def}

\begin{Rem}\label{rem:unique-equiv}
    Given a pseudo-divisor $(\E_1,D_1)$ of a graph $\Gamma$ and a subset $\E_2\subset E(\Gamma)$ such that $\E_2$ is equivalent to $\E_1$, then there is a unique divisor $D_2$ on $\Gamma^{\E_2}$ such that $(\E_2,D_2)$ is equivalent to $(\E_1,D_1)$ (the divisor $D_2$ is defined as $D_2(v):=D_1(v)$ for every $v\in V(\Gamma)$ and $D_2(v_e)=1$ for every $e\in \E_2$). In particular, if $(\E_1,D_1)\sim (\E_2,D_2)$  and $\E_1=\E_2$, then $(\E_1,D_1)=(\E_2,D_2)$.
\end{Rem}
\begin{Rem}
\label{rem:equivalent_cover}
Let $(\E_1,D_1)\geq (\wt{\E}_1,\wt{D}_1)$ be a specialization in $\QS_{v_0}(\Gamma)$.  If $(\E_2,D_2)$ and $(\wt{\E}_2,\wt{D}_2)$ are two pseudo-divisors on $\Gamma$ such that $\wt{\E}_2\subset \E_2$, with $(\E_2,D_2)\sim (\E_1,D_1)$ and $(\wt{\E}_2,\wt{D}_2)\sim (\wt{\E_1},\wt{D}_1)$, then $(\E_2,D_2)\geq (\wt{\E}_2,\wt{D}_2)$.
\end{Rem}

Recall the definition of the set  $\ND(\Gamma)$ in Equation \eqref{eq:ND}.

\begin{Prop}\label{prop:fE} 
The isomorphisms $f$ and $f^{-1}$ take equivalent pseudo-divisors to equivalent pseudo-divisors. Moreover, $f$ induces 
 a weakly cyclic equivalence $f_E\col \ND(\Gamma)\to \ND(\Gamma')$ such that for every pseudo-divisor $(\E,D)\in \QS_{v_0}(\Gamma)$ there exists a unique divisor $D'\in \QS_{v_0'}(\Gamma', f_E(\E))$ for which  $f(\E,D)\sim (f_E(\E), D')$. 
\end{Prop}

\begin{proof}
By Corollary \ref{cor:special}, we can define a bijection $f_E\col \ND(\Gamma)\to \ND(\Gamma')$ as follows.
\begin{enumerate}
    \item For each edge $e\in \ND(\Gamma)$ that does not belong to a special pair of $\Gamma$, we set $f_E(e)$ to be the unique edge of $\Gamma'$ satisfying $\{f_E(e)\} =\epsilon_{\Gamma'}(f(\{e\},D))$ for every divisor $D\in \QS_{v_0}(\Gamma,\{e\})$. (Notice that $\QS_{v_0}(\Gamma,\{e\})$ is not empty, since $e$ is not a bridge of $\Gamma$). 
    \item For each special pair $\{e_1,e_2\}$ of $\Gamma$, we let $f_E(e_1), f_E(e_2)$ to be the edges of $\Gamma'$ such that 
    \[
\{f_E(e_1),f_E(e_2)\} = \epsilon_{\Gamma'}(f(\{e_1,e_2\},D)),
    \] 
    for every divisor $D\in \QS_{v_0}(\Gamma,\{e_1,e_2\})$. (Here there is a choice to be made: a different choice would switch the values $f_E(e_1)$ and $f_E(e_2)$.) Notice that $\{f_E(e_1),f_E(e_2)\}$ is a special pair and hence it is contained in $\ND(\Gamma)$. 
\end{enumerate}

Let us prove that for every $(\E,D)\in \QS_{v_0}(\Gamma)$, there is a unique divisor $D'\in \QS_{v_0'}(\Gamma', f_E(\E))$ such that $f(\E,D)\sim (f_E(\E), D')$. By Remark \ref{rem:unique-equiv}, it suffices to prove that $\epsilon_{\Gamma'}(f(\E,D))$ and $f_E(\E)$ are equivalent subsets of $E(\Gamma')$. Set $\E':=\epsilon_{\Gamma'}(f(\E,D))$. For each subset $\E_0\subset \E$, there exists a divisor $D_0\in \QS_{v_0}(\Gamma,\E_0)$ such that $(\E_0,D_0)\leq (\E,D)$ (see Remark \ref{rem:elementary}). Moreover, we have $f(\E_0,D_0)\leq f(\E,D)$, and hence $\epsilon_{\Gamma'}(f(\E_0,D_0))\subset \epsilon_{\Gamma'}(f(\E,D))=\E'$. Thus the following conditions hold
\begin{enumerate}
    \item if an edge $e\in \E$ does not belong to a special pair, then $f_E(e)\in \E'$.
    \item if $\{e_1,e_2\}\subset \E$ is a special pair, then $f_E(\{e_1,e_2\})\subset \E'$.
    \item if an edge $e_1\in \E$ belongs to a special pair $\{e_1,e_2\}$ with $e_2$ not in $\E$, then either $f_E(e_1)\in \E'$ or $f_E(e_2)\in \E'$, but  $\{f_E(e_1),f_E(e_2)\}\not\subset\E'$. Moreover, $\{f_E(e_1),f_E(e_2)\}$ is a special pair of $\Gamma'$.
\end{enumerate} 
This concludes the proof that $\epsilon_{\Gamma'}(f(\E,D))$ and $f_E(\E)$ are equivalent.

Next, we prove that $f_E$ is a weakly cyclic equivalence. By Remark \ref{rem:weak_cyclic_equiv_trees} it is enough to prove that $f_E$ and $f_E^{-1}$ take maximally nondisconnecting subsets to maximally nondisconnecting subsets. Let $\E\subset E(\Gamma)$ be a maximally nondisconnecting subset. Then there exists exactly one divisor $D\in \QS_{v_0}(\Gamma,\E)$ (see Remark \ref{rem:AP_46}). We also have that $(\E, D)$ is maximal in $\QS_{v_0}(\Gamma)$. We set $(\E',D'):=f(\E,D)$. Then $(\E',D')$ is maximal in $\QS_{v_0'}(\Gamma')$, which implies that $\E'$ is a maximally nondisconnecting subset (see Remark \ref{rem:AP_46}). Since $f_E(\E)$ and $\E'$ are equivalent, we have that $f_E(\E)$ is also a maximally nondisconnecting subset.
The number of spanning trees of $\Gamma$ and $\Gamma'$ is equal to the number of maximal elements of $\QS_{v_0}(\Gamma)$ and $\QS_{v'_0}(\Gamma')$, respectively, hence they are equal, because $f$ is an isomorphism. Since the number of spanning trees of $\Gamma$ and $\Gamma'$ are the same, it follows that $f_E^{-1}$ also takes maximally nondisconnecting subsets to maximally nondisconnecting subsets. This concludes the proof that $f_E$ is a weakly cyclic equivalence.

Now we will prove that $f$ and $f^{-1}$ take equivalent pseudo-divisors to equivalent pseudo-divisors. We proceed by induction on the rank of a pseudo-divisor. Let $(\E_1,D_1)$ and $(\E_2,D_2)$ be two equivalent pseudo-divisors of rank $k$ on $\Gamma$. 
 If $k=0$, that is, if $|\E_1|=|\E_2|=0$,  then $D_1=D_2$ and hence $f(\E_1,D_1)=f(\E_2,D_2)$ and we are done. The same reasoning holds for $f^{-1}$.

  By the induction hypothesis, $f$ and $f^{-1}$ send equivalent pseudo-divisors of rank strictly less than $k$ to equivalent pseudo-divisors. We will prove the induction step only for $f$. The reasoning  for $f^{-1}$ is similar. It is enough to prove the result for $\E_1=\E\cup\{e_1\}$ and $\E_2=\E\cup\{e_2\}$, for some $\E\subset E(\Gamma)$ and for some special pair $\{e_1,e_2\}$ of $\Gamma$ such that $\E\cap\{e_1,e_2\}=\emptyset$. Let $s$ and $t$ be the end-vertices of $e_1$ and $e_2$. Define 
\begin{align*}
    D_s :=& D_1 - v_{e_1}+s=D_2-v_{e_2}+s\\
    D_t :=& D_1 - v_{e_1}+t=D_2-v_{e_2}+t.
\end{align*}
Notice that we have
\begin{equation}\label{eq:compare}
D_s(s)=D_t(s)+1 \quad\text{ and } \quad D_s(t)=D_t(t)-1.
\end{equation}
By Remark \ref{rem:elementary}, we have that 
\begin{align*}
(\E,D_s)\leq (\E_1,D_1), \; \; & (\E,D_s)\le(\E_2,D_2),\\
(\E,D_t)\leq (\E_1,D_1), \; \; &  (\E,D_t)\leq (\E_2,D_2).
\end{align*}
In particular, the set $\{(\E_1,D_1),(\E_2,D_2),(\E,D_s),(\E,D_t)\}$ is a poset isomorphic to the poset $\PosFix$ in Definition \ref{def:fix}. Therefore, the image of this set via $f$ must be one of the images described in Proposition \ref{prop:P0}. Set $(\E_1',D_1')=f(\E_1,D_1)$ and $(\E_2',D_2')=f(\E_2,D_2)$.

By contradiction, assume that we are in the situation described in item (2) of Proposition \ref{prop:P0}. 
This implies that there exist parallel edges $\{e_1',e_2'\}$, a subset $\E'\subset E(\Gamma')\setminus\{e_1',e_2'\}$ and a divisor $D'$ on $\Gamma'^{\E'}$ such that  $f(\E,D_s) = (\E'\cup\{e_1'\},D'+v_{e_1'})$ and $f(\E,D_t)=(\E'\cup\{e_2'\},D'+v_{e_2'})$. By induction hypothesis (recall that $|\E|=k-1$), we have that $\E'\cup\{e_1'\}$ and $\E'\cup\{e_2'\}$ are both equivalent to $f_E(\E)$, hence $\E'\cup\{e_1'\}$ and $\E'\cup\{e_2'\}$ are equivalent subsets of $E(\Gamma')$. We deduce that $\{e_1',e_2'\}$ is a special pair and hence $(\E'\cup\{e_1'\},D'+v_{e_1'})$ and $(\E'\cup\{e_2'\},D'+v_{e_2'})$ are equivalent pseudo-divisors of $\Gamma'$. However, the ranks of $(\E'\cup\{e_1'\},D'+v_{e_1'})$ and $ (\E'\cup\{e_2'\},D'+v_{e_2'})$ are equal to $k-1$, so by the induction hypothesis we get that $f^{-1}(\E'\cup\{e_1'\},D'+v_{e_1'})=(\E,D_s)$ and $f^{-1}(\E'\cup\{e_2'\},D'+v_{e_2'})=(\E,D_t)$ are equivalent pseudo-divisors of $\Gamma$. This implies that $D_s(v)=D_t(v)$ for every $v\in V(\Gamma)$, which contradicts Equation \eqref{eq:compare}.

We deduce that we are in the situation described in item (1) of Proposition \ref{prop:P0}. Then there exist parallel edges $e_1',e_2'$ of $\Gamma'$, a subset $\E'\subset E(\Gamma')\setminus \{e'_1,e'_2\}$ of $\Gamma'$, and a divisor $D'\in \Gamma'^{\E'}$ such that $(\E_1',D'_1)=(\E'\cup\{e_1'\},D'+v_{e'_1})$ and $(\E_2',D'_2)=(\E'\cup\{e_2'\},D'+v_{e'_2})$. Hence $D'_1(v)=D'_2(v)$ for every $v\in V(\Gamma)$, and so $D'_1$ and $D'_2$ satisfy Condition (1) of Definition \ref{def:sim}. Moreover:
\begin{enumerate}
    \item the subsets $\E'\cup\{e_1'\}$ and $f_E(\E\cup\{e_1\})$ of $E(\Gamma')$ are equivalent, by construction.
     \item the subsets $\E'\cup\{e_2'\}$ and $f_E(\E\cup\{e_2\})$ of $E(\Gamma)$ are equivalent, by construction.
     \item the subsets $f_E(\E\cup\{e_1\})$ and $f_E(\E\cup\{e_2\})$ of $E(\Gamma')$ are equivalent, since $f_E$ sends special pairs to special pairs.
\end{enumerate}
This implies that $\E'_1=\E'\cup\{e'_1\}$ and $\E'_2=\E'\cup\{e_2'\}$ are equivalent, and 
hence $f(\E_1,D_1)\sim f(\E_2,D_2)$, concluding the proof. 
\end{proof}

\begin{Def}
Let $f_E$ be as in Proposition \ref{prop:fE}. We let $h_f\col \QS_{v_0}(\Gamma)\to \QS_{v'_0}(\Gamma')$ be the function taking a pseudo-divisor $(\E,D)\in \QS_{v_0}(\Gamma)$ to
\[
h_f(\E,D):=(f_E(\E),D'),
\]
where $D'$ is the unique divisor in $\QS_{v_0'}(\Gamma', f_E(\E))$ such that $f(\E,D)\sim (f_E(\E),D')$ (see Proposition \ref{prop:fE}). 
\end{Def}

By definition, for every pseudo-divisor $(\E,D)\in \QS_{v_0}(\Gamma)$ we have 
\begin{equation}\label{eq:magic}
\epsilon_{\Gamma'}(h_f(\E,D))=f_E(\E).
\end{equation}

\begin{Prop}\label{prop:hf} 
The map $h_f\col \QS_{v_0}(\Gamma)\to \QS_{v'_0}(\Gamma')$ is an isomorphism of ranked posets.
\end{Prop}
\begin{proof}
Let us prove that $h_f$ is a bijection.
We begin by proving that $h_f$ is injective. Assume that $h_f(\E_1,D_1)=h_f(\E_2,D_2)$ for some pseudo-divisors $(\E_1,D_1)$ and $(\E_2,D_2)$ on $\Gamma$. Since $\epsilon_{\Gamma'}(h_f(\E_i,D_i))= f_E(\E_i)$, we have that $f_E(\E_1)=f_E(\E_2)$ which implies that $\E_1=\E_2=:\E$ (recall that $f_E$ is a bijection, see Proposition \ref{prop:fE}). Writing $(\E_i',D_i'):=f(\E,D_i)$ for $i=1,2$, we have that 
\[
(\E_1',D_1')\sim h_f(\E,D_1) = h_f(\E,D_2) \sim (\E_2',D_2'),
\]
hence $(\E_1',D_1')\sim (\E_2',D_2')$.  By Proposition \ref{prop:fE}, we have that 
\[
(\E,D_1) = f^{-1}(\E_1',D_1')\sim f^{-1}(\E_2',D_2') = (\E,D_2),
\]
and hence $D_1=D_2$ by Remark \ref{rem:unique-equiv}. This finishes the proof of the injectivity of $h_f$. Since $\QS_{v_0}(\Gamma)$ and $\QS_{v'_0}(\Gamma')$ are finite sets of the same cardinality, it follows that $h_f$ is bijective.\par 

Let us prove that $h_f$ is a morphism of ranked posets. It is clear that $h_f$ preserves the rank of pseudo-divisors. Assume that $(\E_1,D_1)\geq (\E_2,D_2)$ in $\QS_{v_0}(\Gamma)$. In particular, $\E_2\subset \E_1$. We have that
\begin{align*}
&f(\E_1,D_1)\geq f(\E_2,D_2), \\
& h_f(\E_i,D_i)\sim f(\E_i,D_i),\\
& \epsilon_{\Gamma'}(h_f(\E_2,D_2))=f_E(\E_2)\subset f_E(\E_1)=\epsilon_{\Gamma'}(h_f(\E_1,D_1)).
\end{align*}
Thus $h_f(\E_1,D_1)\geq h_f(\E_2,D_2)$ by Remark \ref{rem:equivalent_cover}, concluding the proof that $h_f$ is a morphism of ranked posets. \par 

Using the same reasoning, we also have that $h_{f^{-1}}\col \QS_{v'_0}(\Gamma')\to \QS_{v_0}(\Gamma)$ is a morphism of ranked posets. 
It remains to prove that $h_{f^{-1}}$ is the inverse of $h_f$. Fix $(\E',D')=h_f(\E,D)$. We have the following equivalences
\begin{enumerate}
    \item $h_{f^{-1}}(\E',D')\sim f^{-1}(\E',D')$, by the definition of $h_{f^{-1}}$.
    \item $f^{-1}(\E',D')\sim f^{-1}(f(\E,D))=(\E,D)$, because $(\E',D')\sim f(\E,D)$ by the definition of $h_f$ and because $f^{-1}$ takes equivalent divisors to equivalent divisors (see Proposition \ref{prop:fE}).
\end{enumerate}
Therefore $h_{f^{-1}}(\E',D')\sim (\E,D)$. 
By the definition of $h_f$ and $h_{f^{-1}}$, we have $\E'=f_E(\E)$ and $\epsilon_{\Gamma}(h_{f^{-1}}(\E',D'))=f_E^{-1}(\E')$. It follows that
\[
\epsilon_{\Gamma}(h_{f^{-1}}(\E',D'))=f_E^{-1}(\E')=\E=\epsilon_{\Gamma}(\E,D).
\]
 Hence $h_{f^{-1}}(\E',D')= (\E,D)$ by Remark \ref{rem:unique-equiv}. This finishes the proof.
\end{proof}

We now substitute the isomorphism $f\col \QS_{v_0}(\Gamma)\to \QS_{v'_0}(\Gamma')$ with $h_f\col \QS_{v_0}(\Gamma)\to \QS_{v'_0}(\Gamma')$, which is an isomorphism by Proposition \ref{prop:hf}. By Equation \eqref{eq:magic}, this allows us to use the following property:
\begin{equation}\label{eq:substitute}
\epsilon_{\Gamma'}(f(\E,D))=f_E(\E),
\end{equation}
for every pseudo-divisor $(\E,D)\in \QS_{v_0}(\Gamma)$.

\begin{Lem}\label{lem:tree}
Assume that $\Gamma$ is a tree and let $v_0$ be a vertex of $\Gamma$. 
Let $D$ be the divisor on $\Gamma$ such that 
\[
D(v)=\begin{cases}
0& \text{ if } v\neq v_0\\
-1& \text{ if } v=v_0\\
\end{cases}
\]
for every $v\in V(\Gamma)$.
Then $D$ is the unique element of $\QS_{v_0}(\Gamma)$.
\end{Lem}

\begin{proof}
By Remark \ref{rem:AP_46}, the poset $\QS_{v_0}(\Gamma)$ is a singleton. So it is enough to prove that the divisor $D$ given by the formula in the statement is $v_0$-quasistable. By Equation \eqref{eq:integer} we have 
$\beta_{\Gamma,D}(V)
= D(V)- g_V+1$ for every  hemisphere $V\subset V(\Gamma)$. Since $\Gamma$ is a tree, we have $g_V = 0$ for every hemisphere $V\subset V(\Gamma)$. We also have
\[
D(V)=\begin{cases}
0& \text{ if } v_0\notin V,\\
-1& \text{ if } v_0\in V.\\
\end{cases}
\]
It follows that 
\[
\beta_{\Gamma,D}(V) = \begin{cases}
0 & \text{ if } v_0\in V, \\
1 & \text{ if } v_0\notin V.
\end{cases}
\]
This proves that $D$ is $v_0$-quasistable. 
\end{proof}

\begin{Lem}
\label{lem:vertex}
Let $v_1$ be a vertex of $\Gamma$ which is not an articulation vertex. Fix a maximally nondisconnecting subset $\mathcal E_1\subset E(\Gamma\setminus \{v_1\})$ of $\Gamma\setminus\{v_1\}$. There is a unique divisor $D_1$ in $\QS_{v_0}(\Gamma,\E_1)$ such that
\begin{equation}\label{eq:valence}
D_1(v_1)=\begin{cases}
\val(v_1)-1& \text{ if }v_1\neq v_0\\
\val(v_1)-2& \text{ if }v_1=v_0.
\end{cases}
\end{equation}
Moreover, for each $S\subsetneqq E(v_1)$ there exists a unique $D_S$ in $\QS_{v_0}(\Gamma,\E_1\cup S)$ such that $(\E_1,D_1)\leq (\E_1\cup S,D_S)$.
\end{Lem}
\begin{proof}
By Remark \ref{rem:AP_46} we can assume that $\E_1=\emptyset$ and $\Gamma\setminus\{v_1\}$ is a tree.

For $e\in E(v_1)$, we set $S_e:=E(v)\setminus \{e\}$.
By Lemma \ref{lem:tree}, there exists a unique divisor $D_{S_e}\in \QS_{v_0}(\Gamma,S_e)$ and we have $D_{S_e}(v_0)=-1$ and $D_{S_e}(u)=0$ for every $u\in V(\Gamma)\setminus\{v_0\}$. Note that $S_e$ is a maximally nondisconnecting subset of $\Gamma$ and, vice-versa, any maximally nondisconnecting subset of $\Gamma$ is of the form $S_e$ for some $e\in E(v_1)$. In particular, by Remark \ref{rem:AP_46} we have that $\{(S_e,D_{S_e})\}_{e\in E(v_1)}$ is the set of all  maximal elements of $\QS_{v_0}(\Gamma)$.

Set $D_1:=D_{S_e}-\sum_{\widetilde{e}\in S_e} v_{\widetilde{e}} + (\val(v)-1)v$. We have that $(\emptyset, D_1)$ is a specialization of $(S_e,D_{S_e})$ and, since $(S_e,D_{S_e})$ is $v_0$-quasistable, it follows that $(\emptyset, D_1)$ is $v_0$-quasistable as well (see Remark \ref{rem:AP_46}). This proves the existence of $D_1$. Note that $D_1$ is independent of the choice of $e\in E(v_1)$. \par 

On the other hand, if $\widetilde{D}_1$ is another such divisor, then $(\emptyset,\widetilde{D}_1)$ is smaller than a maximal element $(S_e,D_{S_e})$ for some $e\in E(v)$. By Lemma \ref{lem:tree} and Equation \ref{eq:valence}, we can write $\widetilde{D}_1(v_1)=D_{S_e}(v_1)+\val(v_1)-1$. Since  $|S_e|=\val(v_1)-1$, the only possible way for $(S_e,D_{S_e})$ to specialize to $(\emptyset,\widetilde{D}_1)$ is if $\widetilde{D}_1=D_{S_e}-\sum_{\widetilde{e}\in S_e} v_{\widetilde{e}} + (\val(v)-1)v$. This means that $\widetilde{D}_1=D_1$ and finishes the proof of the first statement.

Fix $S\subsetneqq E(v_1)$. There exists $e\in E(v_1)$ such that $S\subset S_e$. The divisor $D_S:=D_{S_e}-\sum_{\widetilde{e}\in S_e\setminus S}v_{\widetilde{e}}+|S_e\setminus S|v$ is independent of the choice of $e$ and satisfies $(\emptyset,D_1)\leq (S,D_S)$. Moreover, we have $D_S\in \QS_{v_0}(\Gamma,S)$, because the $v_0$-quasistable  pseudo-divisor $(S_e, D_{S_e})$ specializes to $(S, D_s)$ (see Remark \ref{rem:AP_46}). 

We claim that the divisor $(S,D_S)$ is the unique divisor in $\QS_{v_0}(\Gamma, S)$ such that   $(\emptyset,D_1)\le (S,D_S)$. Indeed, assume that $\widetilde{D}_S$ is another divisor in $\QS_{v_0}(\Gamma,S)$ such that $(\emptyset,D_1)\leq (S, \widetilde{D}_S)$. Then, there exists a maximal pseudo-divisor $(S_e,D_{S_e})$ that is greater than $(S,\widetilde{D}_S)$. This implies that \begin{equation}\label{eq:less}
\widetilde{D}_S(v_1)\leq D_{S_e}(v_1)+ \val(v_1)-|S|-1.
\end{equation}
On the other hand, since $(S,\widetilde{D}_S)$ is $v_0$-quasistable and greater than $(\emptyset,D_1)$, we must have that $\widetilde{D}_S(v_1)\geq D_{S_e}(v_1)+\val(v_1)-|S|-1$, and hence equality holds in Equation \eqref{eq:less}.    This implies that $\widetilde{D}_S=D_S$.
\end{proof}

\begin{Lem}\label{lem:bond}
 Let $V$ be a hemisphere of $\Gamma$. Let $\E_1\subset E(V,V)$ and $\E_2\subset E(V^c,V^c)$ be maximally nondisconnecting subsets of $\Gamma(V)$ and $\Gamma(V^c)$. Set $\E=\E_1\cup \E_2$. Let $D$ be a divisor in $\QS_{v_0}(\Gamma,\E)$ such that for each subset $S\subsetneqq E(V,V^c)$ there exists a unique $D_S\in \QS_{v_0}(\Gamma,\E\cup S)$ such that $(\E,D)\leq (\E\cup S,D_S)$. Then there exists a vertex $v_1\in V(\Gamma)$ that is incident to all edges in $E(V,V^c)$.
\end{Lem}

\begin{proof}
By Remark \ref{rem:AP_46}, we can assume that $\E=\emptyset$. This implies that $\Gamma(V)$ and $\Gamma(V^c)$ are trees. \par 

Consider $S\subsetneqq E(V,V^c)$ and the unique divisor $D_S$ of the statement. Let us prove that there exists only one specialization $(S,D_S)\to (\emptyset,D)$.
Assume, by contradiction, that there are two different specializations $\iota_1,\iota_2\col (S,D_S)\to (\emptyset,D)$. This implies that  we have distinct edges $e_1$ and $e_2$ in $S$ such that  $\iota_1(v_{e_1}),\iota_2(v_{e_2})\in V$ and $\iota_1(v_{e_2}),\iota_2(v_{e_1})\in V^c$ (note that the degrees of $(\iota_1)_*(D_S)=D$ and $(\iota_2)_*(D_S)=D$ in $V$ and $V^c$ are the same).\par 
Let $S_0:=\{e_1,e_2\}$. We can consider the specialization $\iota'_i\col \Gamma^{S}\ra \Gamma^{S_0}$ giving rise to a factorization
\[
\iota_i\col (S,D_S)\ra(S_0,(\iota'_i)_*(D_S))\stackrel{j_i}{\ra}(\emptyset,D).
\] 
Then $(\iota'_i)_*(D_S)=D_{S_0}$, by the uniqueness of $D_{S_0}$. Hence we get the specializations $j_1,j_2\colon (S_0, D_{S_0})\to (\emptyset, D)$. These specializations must be distinct because $j_1(v_{e_1}),j_2(v_{e_2})\in V$ and $j_1(v_{e_2}),j_2(v_{e_1})\in V^c$. Let $t_{e_i}$ and $s_{e_i}$ be the end-vertices of $e_i$, with $t_{e_i}\in V$. Thus
\begin{align*}
D=(j_1)_*(D_{S_0})&=D_{S_0}+t_{e_1}+s_{e_2}-v_{e_1}-v_{e_2}\\
D=(j_2)_*(D_{S_0})&=D_{S_0}+s_{e_1}+t_{e_2}-v_{e_1}-v_{e_2}.
\end{align*}
It follows that $t_{e_1}+s_{e_2}=s_{e_1}+t_{e_2}$, hence $t_{e_1}=t_{e_2}$ and $s_{e_1}=s_{e_2}$ (this means that $e_1,e_2$ are parallel edges).
Set $S_1:=\{e_1\}$. We see that the two pseudo-divisors $(S_1, D_{S_0}-v_{e_2}+t_{e_2})$ and $(S_1, D_{S_0}-v_{e_2}+s_{e_2})$ are both greater than $(\emptyset,D)$, which  contradicts the uniqueness of $D_{S_1}$. This proves that there exists a unique specialization $( S,D_S)\to (\emptyset,D)$, which we will denote by $\iota_S\colon (S,D_S)\to (\emptyset,D)$. \par 

For every $e\in E(V,V^c)$, we set $u_e:=\iota_{\{e\}}(v_e)\in V(\Gamma)$. Let us prove that $\iota_S(v_e)=u_e$, for every  $e\in E(V,V^c)$ and every subset $S\subsetneqq E(V,V^c)$ containing $e$. In fact, for every such edge $e$ and subset $S$, let $\iota'\colon ( S,D_S)\to (\{e\},D_{\{e\}})$ be the specialization factoring $\iota_S$ (note that this is unique by the uniqueness of $\iota_S$), i.e., such that $\iota_S=\iota_{\{e\}}\circ \iota'$. We see that $\iota_S(v_e)=\iota_{\{e\}}(v_e)=u_e$, as wanted. \par

Now we claim that, if $u_{e_0}\in V$ for some $e_0\in E(V,V^c)$, then $u_e\in V$ for every $e\in E(V,V^c)$. By contradiction, assume that there are edges $e_1,e_2$, such that $u_{e_1}\in V$ and $u_{e_2}\in V^c$. Set $S_i=E(V,V^c)\setminus \{e_i\}$ for $i=1,2$. Since
\begin{align*}
(\iota_{S_i})_*(D_{S_i})(V)&=D_{S_i}(V)+|\{e\in S_i, u_e\in V\}|
\end{align*}
and $(\iota_{S_i})_*(D_{S_i})(V)=D(V)$, we have that 
\[
D_{S_1}(V) = D_{S_2}(V) +1.
\]
However, since $S_i$ is a maximally nondisconnecting subset of $\Gamma$, Lemma \ref{lem:tree} implies that $D_{S_1}(V)=D_{S_2}(V)$, giving rise to a contradiction. This proves the claim.\par 

Finally, let us prove that $u_{e_1}=u_{e_2}$ for every $e_1,e_2\in E(V,V^c)$. As before, set $S_i=E(V,V^c)\setminus\{e_i\}$. By Lemma \ref{lem:tree}, we have that $D_{S_1}(v)=D_{S_2}(v)$ for every $v\in V(\Gamma)$, and 
\[
D(v)=D_{S_i}(v)+|\{e\in S_i;u_e=v\}|
\]
for every $v\in V(\Gamma)$. Hence, taking $v=u_{e_1}$, we have that
\[
|\{e\in S_1;u_e=u_{e_1}\}|=|\{e\in S_2;u_e=u_{e_1}\}|
\]
which implies that $u_{e_1}=u_{e_2}$. The conclusion is that, if we set $v_1:=u_e$ for some (every) edge $e\in E(V,V^c)$, then $v_1$ is incident to every $e\in E(V,V^c)$. 
\end{proof}

\begin{Thm}\label{thm:main1-biconnected}
Let $\Gamma$ and $\Gamma'$ be  biconnected pure graphs. The posets $\QS(\Gamma)$ and $\QS(\Gamma')$ are isomorphic if and only if there is an isomorphism between $\Gamma$ and $\Gamma'$.
\end{Thm}

\begin{proof}
Assume that $\QS(\Gamma)$ and $\QS(\Gamma')$ are isomorphic. Recall that we have identifications $\QS(\Gamma)\cong \QS_{v_0}(\Gamma)$ and $\QS(\Gamma')\cong \QS_{v'_0}(\Gamma')$ for $v_0\in V(\Gamma)$ and $v'_0\in V(\Gamma')$. Recall that we are given an isomorphism of posets $f\col \QS_{v_0}(\Gamma)\ra\QS_{v'_0}(\Gamma')$. 
Since $\Gamma$ and $\Gamma'$ are biconnected, they have no bridges, hence 
by Proposition \ref{prop:fE}, there is a cyclic equivalence $f_E\col E(\Gamma)\ra E(\Gamma')$.

Assume that $\Gamma$ has only one vertex. Then $\Gamma$ has at most one edge, hence $\Gamma'$ is isomorphic to $\Gamma$ because $f_E$ is a cyclic equivalence. The same argument holds if $\Gamma'$ has only one vertex.
Assume that $\Gamma$ has two vertices. Since $f_E$ is a cyclic equivalence and since every set of two edges of $\Gamma$ is a cycle, we must have that $\Gamma'$ also has two vertices and the same number of edges. So $\Gamma$ and $\Gamma'$ are isomorphic. The same argument holds if $\Gamma'$ has two vertices. So, we can assume that $\Gamma$ and $\Gamma'$ have at least three vertices.\par 

First we observe that if $S'$ is a subset of $E(\Gamma')$, then there exists at most one vertex $v'_1$ such that $E(v'_1)=S'$. Indeed, if there are distinct vertices $v'_1,v'_2$ such that $E(v'_1)=E(v'_2)=S'$, then either $\Gamma'$ is disconnected or $V(\Gamma')=\{v_1', v'_2\}$, which is a contradiction. \par 

To prove that $\Gamma$ and $\Gamma'$ are isomorphic, it is sufficient to prove that for every $v_1\in V(\Gamma)$ there exists a unique $v'_1\in V(\Gamma')$ such that $E(v'_1)=f_E(E(v_1))$. By the above observation, it is sufficent to prove that for every $v_1$, there exists a $v'_1\in V(\Gamma')$ such that $E(v'_1)=f_E(E(v_1))$. \par 

Fix $v_1\in V(\Gamma)$. Since $\Gamma$ is biconnected, we have that $v_1$ is not an articulation vertex. Let $\E_1\subset E(\Gamma\setminus\{v_1\})$ be a maximally nondisconnecting subset of $\Gamma\setminus \{v_1\}$. Let $D_1\in \QS_{v_0}(\Gamma,\E_1)$ be as in Lemma \ref{lem:vertex}. The same lemma states that  for each $S\subsetneqq E(v)$, there exists a unique $D_S\in \QS_{v_0}(\Gamma,\E_1\cup S)$ such that $(\E_1,D_1)\leq (\E_1\cup S,D_S)$.\par 

Since $E(v_1)$ is a bond of $\Gamma$ (recall that $\Gamma$ is biconnected) and $f_E$ is a cyclic equivalence, by Remark \ref{rem:cyclic_equiv_trees} we have that $f_E(E(v_1))$ is also a bond of $\Gamma'$, that is, there exists a hemisphere $V'\subset V(\Gamma')$ such that $f_E(E(v_1))=E(V',V'^c)$. Set $(\E'_1,D'_1)=f(\E_1,D_1)$. Since $f$ is an isomorphism and $\epsilon_{\Gamma'}\circ f = f_E\circ \epsilon_\Gamma$ (recall Equation \eqref{eq:substitute}), we have that for each $S'\subsetneqq E(V', V'^c)$ there exists a unique $D'_{S'}\in \QS_{v_0'}(\Gamma', \E'_1\cup S')$ such that $ (\E'_1,D'_1)\leq (\E'_1\cup S',D'_{S'})$. By Lemma \ref{lem:bond}, there exists a vertex $v'_1$ such that $E(V',V'^c)\subset E(v'_1)$. However, $\Gamma'$ is biconnected, that means that either $V'=\{v'_1\}$ or $V'^c=\{v'_1\}$, otherwise $v'_1$ would be an articulation vertex of $\Gamma'$. This means that $f_E(E(v_1))=E(v'_1)$ and we are done.\par 

If $\Gamma$ and $\Gamma'$ are isomorphic it is clear that $\QS(\Gamma)$ and $\QS(\Gamma')$ are isomorphic as well.
\end{proof}

\begin{Def}\label{def:biconnected_division}
    Let $v_0$ be an articulation vertex of a graph $\Gamma$. A pair of connected subgraphs $(\Gamma_1, \Gamma_2)$, with $E(\Gamma_1)\ne \emptyset$ and $E(\Gamma_2)\ne \emptyset$, are called     
    \emph{a split of $\Gamma$ with respect to $v_0$} if 
    \[
    \begin{array}{ll}
    V(\Gamma_1)\cap V(\Gamma_2)=\{v_0\}, & V(\Gamma_1)\cup V(\Gamma_2)=V(\Gamma),\\E(\Gamma_1)\cap E(\Gamma_2)=\emptyset, & E(\Gamma_1)\cup E(\Gamma_2)=E(\Gamma).
    \end{array}
    \]
\end{Def}

\begin{Rem}\label{rem:v0-induced}
It easy to check that, given an articulation vertex $v_0$, there always exists a split $(\Gamma_1,\Gamma_2)$ of $\Gamma$ with respect to $v_0$.  
Notice that the connected componets of $\Gamma_1\setminus\{v_0\}$ and $\Gamma_2\setminus \{v_0\}$ form a partition of the connected components of $\Gamma\setminus\{v_0\}$. Moreover, the biconnected components of $\Gamma_1$ and $\Gamma_2$ are biconnected components of $\Gamma$ and, conversely, every biconnected component of $\Gamma$ is a biconnected component of one between $\Gamma_1$ or $\Gamma_2$.
\end{Rem}

\begin{Prop}\label{prop:2-decomposition}
Let $\Gamma$ be a pure graph and 
$v_0$ an articulation vertex of  $\Gamma$. 
Let $(\Gamma_1,\Gamma_2)$ be a split of $\Gamma$ with respect to $v_0$. We have an isomorphism
\[
\sigma\col \QS_{v_0}(\Gamma_1)\times\QS_{v_0}(\Gamma_2)\stackrel{\cong}{\ra}\QS_{v_0}(\Gamma)
\]
taking a pair of pseudo-divisors $((\E_1,D_1),(\E_2,D_2))$ to $(\E_1\cup\E_2,D_1+D_2+v_0)$. Moreover, if $e\in E(\Gamma_1)$ and $(\E,D)\to (\E\setminus\{e\},\overline{D})$ is an elementary specialization in $\QS_{v_0}(\Gamma)$ over $e$, then $\sigma^{-1}(\E\setminus \{e\},\overline{D})=((\E_1\setminus\{e\},\overline{D}_1),(\E_2,D_2))$, where $(\E_1\setminus\{e\},\ol{D}_1)$ is an elementary specialization of $(\E_1,D_1)$ in $\QS_{v_0}(\Gamma_1)$ over $e$.
\end{Prop}

\begin{proof}
Let $((\E_1,D_1)$, $(\E_2,D_2))$ and $(\E,D):=(\E_1\cup\E_2,D_1+D_2+v_0)$ be as in the statement.
Since $\Gamma$ is pure, we have $g=g_\Gamma=g_{\Gamma_1}+g_{\Gamma_2}$. Hence the degree of $D$ is $g-1$. Let us see that $\sigma$ is well-defined, proving that $(\E,D)\in \QS_{v_0}(\Gamma)$. We use Remark \ref{rem:hemi} and Equation \eqref{eq:integer}. Let $V\subset V(\Gamma^\E)$ be a hemisphere. Assume that $v_0\notin V$. Since $V$ is a hemisphere we have that $V\subset V(\Gamma_i)\setminus \{v_0\}$ for some $i=1,2$. In this case, we can assume without loss of generality that $i=1$, and we have that 
\[
\beta_{\Gamma^\E,D}(V)=\beta_{\Gamma_1^{\E_1},D_1}(V\cap V(\Gamma_1)) > 0.
\] 
On the other hand, if $v_0\in V$, then we have that
\[
\beta_{\Gamma^\E,D}(V)=\beta_{\Gamma_1^{\E_1},D_1}(V\cap V(\Gamma_1)) +\beta_{\Gamma_2^{\E_2},D_2}(V\cap V(\Gamma_2))\geq 0.
\]
This proves that $(\E,D)\in \QS_{v_0}(\Gamma)$ and hence the function $\sigma$ is well-defined.\par 

Given a specialization $(\E_i,D_i)\ra (\E'_i,D'_i)$ in $\QS_{v_0}(\Gamma_i)$ for every $i=1,2$, we have an induced specialization $\sigma((\E_1,D_1),(\E_2,D_2))\ra\sigma((\E'_1,D'_1),(\E'_2,D'_2))$ via the inclusions $\E_i\subset E(\Gamma)$ and $\E'_i\subset E(\Gamma)$. This implies that 
$\sigma$ is a morphism of posets.

Let us prove that $\sigma$ is injective. Assume that $\sigma((\E_1,D_1),(\E_2,D_2))=\sigma((\E_1',D_1'),(\E_2',D_2'))$. It is clear that $\E_1=\E_1'$ and $\E_2=\E_2'$. Moreover, it is also clear that for each $i=1,2$ and for each vertex $v\in V(\Gamma_i)\setminus \{v_0\}$, we have that $D_i(v)=D'_i(v)$. Since $D_i$ and $D'_i$ have the same degree, we have that $D_i(v_0)=D_i'(v_0)$ for $i=1,2$. Thus $D_1=D'_1$ and $D_2=D'_2$, as wanted.

Let us prove that $\sigma$ is surjective. Since we already know that $\sigma$ is injective, we need only to prove that the cardinalities of the domain and target of $\sigma$ are the same. The number of elements of $\QS_{v_0}(\Gamma)$ is $2^g$ times the number of spanning trees of $\Gamma$. Since $2^g=2^{g_1}\cdot 2^{g_2}$ and each spanning tree of $\Gamma$ is a union of  spanning trees of $\Gamma_1$ and $\Gamma_2$, the result follows.

 Finally, we show that $\sigma^{-1}$ is a morphism of posets. We start with an elementary specialization $(\E,D)\ra (\ol{\E},\ol{D})$ in $\QS_{v_0}(\Gamma)$. Let us show that $\sigma^{-1}(\ol{\E},\ol{D})\le \sigma^{-1}(\E,D)$. Since every specialization is a composition of elementary specializations, we can assume that $(\E,D)\ra (\ol{\E},\ol{D})$ is elementary. By Remark \ref{rem:elementary}, we can write  $(\ol{\E},\ol{D})=(\E\setminus \{e\},D-v_e+s)$, for some edge $e\in E(\Gamma)$ with end-vertex $s$. Assume that $e\in E(\Gamma_1)$. In particular, $s\in V(\Gamma_1)$. Set $((\E_1,D_1),(\E_2,D_2)):=\sigma^{-1}((\E,D))$. Consider the elementary specialization $(\E_1,D_1)\ra (\ol{\E}_1,\ol{D}_1)$ in $\QS_{v_0}(\Gamma_1)$, where $(\ol{\E}_1,\ol{D}_1)=(\E_1\setminus\{e\},D_1-v_e+s)$. Clearly we have  $\sigma((\ol{\E}_1,\ol{D}_1),(\E_2,D_2))=(\ol{\E},\ol{D})$. This proves that $\sigma^{-1}(\ol{\E},\ol{D})\leq \sigma^{-1}(\E,D)$, as wanted. Notice that we have also proved the last statement of the proposition.
\end{proof}

\begin{Cor}\label{cor:decomposition}
Given a pure graph $\Gamma$, we have an isomorphism
\[
\QS(\Gamma) \cong \prod \QS(\Gamma_i),
\]
where $\Gamma_i$ runs through all biconnected components of $\Gamma$.
\end{Cor}

\begin{proof} 
The result readily follows from Proposition \ref{prop:2-decomposition}.
\end{proof}

We are now ready to prove Theorem \ref{thm:main1}.

\begin{proof}[Proof of Theorem \ref{thm:main1}]
Recall that we have reduced to the case where $\Gamma$ and $\Gamma'$ are pure graphs (recall Proposition \ref{prop:pure-reduction}).

Assume that there is a bijection between the biconnected components of $\Gamma/\bridges(\Gamma)$ and $\Gamma'/\bridges(\Gamma')$ such that the corresponding components are isomorphic. We must prove that $\QS(\Gamma)$ and $\QS(\Gamma')$ are isomorphic. By Remark \ref{rem:iso-bridge}, we need only to show that $\QS(\Gamma/\bridges(\Gamma))$ and $\QS(\Gamma'/\bridges(\Gamma'))$ are isomorphic. This clearly follows from Corollary \ref{cor:decomposition}.

Conversely, assume that $f\col \QS(\Gamma)\ra \QS(\Gamma')$ is an isomorphism. By Remark \ref{rem:iso-bridge} we can assume that $\Gamma$ and $\Gamma'$ have no bridges. Consider the cyclic equivalence $f_E\col E(\Gamma)\to E(\Gamma')$ given by Proposition \ref{prop:fE}. This induces a bijection between the sets of biconnected components of $\Gamma$ and $\Gamma'$. We proceed by induction on the number of biconnected components of $\Gamma$. If $\Gamma$ is biconnected, the result follows from Theorem \ref{thm:main1-biconnected}. 

Assume that $\Gamma$ is not biconnected. Let $v_0$ be an articulation vertex of $\Gamma$. Let $(\Gamma_1,\Gamma_2)$ be a split of $\Gamma$ with respect to $v_0$ (see  Definition \ref{def:biconnected_division}). Let $\Gamma_1'$ and $\Gamma_2'$ be the subgraphs of $\Gamma'$ such that $E(\Gamma'_i)=f_E(E(\Gamma_i))$. Since $f_E$ is a cyclic equivalence, there is an articulation vertex $v'_0$ of $\Gamma'$ such that $(\Gamma'_1, \Gamma'_2)$ is a split of $\Gamma'$ with respect to $v'_0$. Choose identifications $\QS(\Gamma)\cong \QS_{v_0}(\Gamma)$ and $\QS(\Gamma')\cong \QS_{v'_0}(\Gamma')$. Let 
\[
\sigma\col \QS_{v_0}(\Gamma_1)\times \QS_{v_0}(\Gamma_2)\to \QS_{v_0}(\Gamma)\]
\[
\sigma'\col \QS_{v_0'}(\Gamma_1')\times \QS_{v_0'}(\Gamma_2')\to \QS_{v'_0}(\Gamma')
\]
be the isomorphisms of Proposition  \ref{prop:2-decomposition}. Define  
\[
\overline{f}:=\sigma'^{-1}\circ f\circ \sigma \col\QS_{v_0}(\Gamma_1)\times \QS_{v_0}(\Gamma_2) \to \QS_{v_0'}(\Gamma_1')\times \QS_{v_0'}(\Gamma_2'),
\]
and let $\overline{f}_i\col \QS_{v_0}(\Gamma_1)\times \QS_{v_0}(\Gamma_2)\ra \QS_{v'_0}(\Gamma'_i)$ be the composition of $\overline{f}$ with the projection onto the $i$-th factor.

We claim that $\overline{f}_1((\E_1,D_1),(\E_2,D_2))$ is independent of $(\E_2,D_2)$ (and, similarly, $\overline{f}_2((\E_1,D_1),(\E_2,D_2))$ is independent of $(\E_1,D_1)$). The claim allows us to conclude the proof. Indeed, it implies that 
\[
\overline{f}((\E_1,D_1),(\E_2,D_2))=(f_1(\E_1,D_1),f_2(\E_2,D_2)),
\]
where $f_i\col \QS_{v_0}(\Gamma_i)\to \QS_{v'_0}(\Gamma'_i)$ is an isomorphism induced by 
$\overline{f}_i$.
We conclude the proof by the induction hypothesis, using Remark \ref{rem:v0-induced}.

 To prove the claim, let us start with an observation coming from Proposition \ref{prop:2-decomposition}. Let $(\E',D')\to (\E'\setminus\{e'\},\overline{D}')$ be an elementary specialization in $\QS_{v'_0}(\Gamma')$ with $e'\in E(\Gamma'_2)$. Set $((\E_1',D_1'),(\E_2',D_2')):=\sigma'^{-1}(\E',D')$. By Proposition \ref{prop:2-decomposition}, we have that $\sigma'^{-1}(\E'\setminus \{e'\},\overline{D}')=((\E_1',D_1'),(\E_2'\setminus\{e'\},\ol{D}_2')$, where $(\E_2'\setminus\{e'\},\ol{D}_2')$ is an elementary specialization of $(\E_2',D_2')$ in $\QS_{v'_0}(\Gamma'_2)$ over $e'$.

Now, we just note that if $(\E,D)\to (\E\setminus\{e\},\overline{D})$ is an elementary specialization in $\QS_{v_0}(\Gamma)$ over $e$, then $f(\E,D)\to f(\E\setminus\{e\},\overline{D})$ is an elementary specialization in $\QS_{v'_0}(\Gamma')$ over $f_E(e)$. In particular, if $(\E_2,D_2)\to (\E_2\setminus\{e\},\overline{D}_2)$ is an elementary specialization in $\QS_{v_0}(\Gamma_2)$ over $e\in E(\Gamma_2)$, then $\sigma((\E_1,D_1),(\E_2,D_2))\to \sigma((\E_1,D_2),(\E_2\setminus\{e\},\overline{D}_2))$ is an elementary specialization in $\QS_{v_0}(\Gamma)$ over $e$. Then,
\[
f\circ \sigma((\E_1,D_1),(\E_2,D_2))\to f\circ \sigma((\E_1,D_2),(\E_2\setminus\{e\},\overline{D}_2))
\]
is an elementary specialization in $\QS_{v'_0}(\Gamma')$ over $f_E(e)\in E(\Gamma_2')$. By the above observation, we have that $\overline{f}_1((\E_1,D_1),(\E_2,D_2))=\overline{f}_1((\E_1,D_1),(\E_2\setminus\{e\},\overline{D}_2))$. Since $\QS_{v_0}(\Gamma_2)$ is connected and any specialization is a composition of elementary specialization, we have that $\overline{f}_1((\E_1,D_1),(\E_2,D_2))$ is independent of $(\E_2,D_2)$ and we are done.
\end{proof}

\section{Torelli Theorem for tropical curves}

A \emph{metric graph} is a pair $(\Gamma,\ell)$, where $\Gamma=(E(\Gamma),V(\Gamma))$ is a graph and $\ell\col E(\Gamma)\ra \mathbb R_{>0}$ is a function. A tropical curve is a metric space obtained by gluing segments $[0,\ell(e)]$, for every $e\in E(\Gamma)$ at their end-vertices as prescribed by the combinatorial data of the graph. We call $(\Gamma,\ell)$ a \emph{model} of the tropical curve.

  Given a tropical curve $X$ associated to a metric graph $(\Gamma,\ell)$, we say that $(\Gamma,\ell)$ is the \emph{canonical model} of $X$ if $\Gamma$ has no vertices of valence $2$ or if $\Gamma$ is the graph with only one vertex and one edge.  The canonical model of a tropical curve $X$ is unique, and we write $(\Gamma_X,\ell_X)$ for the canonical model of $X$. 
A \emph{bridge} of a tropical curve $X$ is a bridge of the graph $\Gamma_X$. A \emph{biconneted} component of a tropical curve $X$ is the tropical curve with model $(\Gamma',\ell')$, where $\Gamma'$ is a biconnected component of $\Gamma_X$ and $\ell'$ is the restriction of $\ell$ to $\Gamma'$.
  
   A tropical curve has an associated tropical Jacobian $J(X)$, which was first introduced in \cite{MZ}. The tropical Jacobian $J(X)$ has the following structure as a polyhedral complex. For each pseudo-divisor $(\E,D)$ of $\Gamma_X$, let $\mathcal{P}_X(\E,D)=\prod_{e\in \E}[0,\ell(e)]$. For each specialization $(\E,D)\to (\E',D')$ there is an associated face morphism $\mathcal{P}_X(\E',D')\subset \mathcal{P}_X(\E,D)$. Fix $v_0\in V(\Gamma_X)$, and define 
\[
J^{\quasi}_{v_0}(X):=\lim\mathcal{P}_X(\E,D) 
\]
where the colimit is taken through all $(\E,D)\in \QS_{v_0}(\Gamma)$. 
 By \cite[Theorem 5.10]{APPLMS} we have that $J(X)$ and $J^{qs}_{v_0}(X)$ are homeomorphic. The structure of a polyhedral complex for the tropical Jacobian was first described in \cite{ABKS}, and was extended in \cite{APPLMS},  \cite{CPS} and \cite{Poly}.

By Proposition \ref{prop:one-QD}, we have that $J^{qs}_{v_0}(X)$ does not depends on $v_0$, so we denote it by $J^{qs}_{v_0}(X)$. 

The following result is a corollary of Theorem \ref{thm:main1}.

\begin{Thm}
\label{thm:main2}
    Let $X$ and $X'$ be tropical curves without bridges such that $J(X)$ and $J(X')$ are isomorphic as polyhedral complexes (with the structure of polyhedral complexes given by $\QS(\Gamma_X)$ and $\QS(\Gamma_{X'})$). Then, there is a bijection between the biconnected components of $X$ and $X'$ such that corresponding components are isomorphic.
\end{Thm}

\begin{proof}
An isomorphism $f_J\col J^{\quasi}(X)\ra J^{\quasi}(X')$ induces an isomorphism between $f\col \QS(\Gamma_X)\ra \QS(\Gamma_{X'})$ and hence, by Theorem \ref{thm:main1}, also isomorphisms between the biconnected components of $\Gamma_X$ and of $\Gamma_{X'}$. In particular, if $e\in E(\Gamma_X)$ is an edge not contained in any special pair and $D\in \QS(\Gamma, \{e\})$, we have that $f(\{e\},D)=(f_E(e), D')$ for some $D'\in \QS(\Gamma', \{f_E(e)\})$. Moreover, we also have that $f_J(\mathcal P_X(\{e\},D))=\mathcal P_{X'}(f_E(e),D')$. Since $\mathcal P_X(\{e\},D)$ is a segment with length $\ell(e)$, we have that $\ell(e)=\ell(f_E(e))$. If $\{e_1,e_2\}$ is a special pair, we have that $f_J(\mathcal P_{X}(\{e_1,e_2\},D))=\mathcal P_{X'}(\{f_E(e_1),f_E(e_2)\},D')$, which means that $\{\ell(e_1),\ell(e_2)\}=\{\ell(f_E(e_1)),\ell(f_E(e_1))\}$. Since $e_1,e_2$ are conjugated by an automorphism of $\Gamma_X$ and $f_E(e_1)$, $f_E(e_2)$ are conjugated by an automorphism of $\Gamma_{X'}$, we have that $X$ and $X'$ have isomorphic biconnected components.
\end{proof}

\bibliographystyle{amsalpha}
\bibliography{bibliography}

\end{document}